\documentclass[pdflatex,sn-mathphys-num]{sn-jnl}

\usepackage{graphicx}%
\usepackage{multirow}%
\usepackage{amsmath,amssymb,amsfonts}%
\usepackage{bbold}
\usepackage{amsthm}%
\usepackage{mathrsfs}%
\usepackage[title]{appendix}%
\usepackage{xcolor}%
\usepackage{textcomp}%
\usepackage{manyfoot}%
\usepackage{booktabs}%
\usepackage{algorithm}%
\usepackage{algorithmicx}%
\usepackage{algpseudocode}%
\usepackage{listings}%
\usepackage{anyfontsize}

 \usepackage{enumitem}
\usepackage{xcolor}
	\definecolor{dark-red}{rgb}{0.7,0.25,0.25}
	\definecolor{dark-blue}{rgb}{0.15,0.15,0.55}
	\definecolor{medium-blue}{rgb}{0,0,.8}
	\definecolor{DarkGreen}{RGB}{0,150,0}
	\definecolor{rho}{named}{red}
	\hypersetup{
	   colorlinks, linkcolor={purple},
	   citecolor={medium-blue}, urlcolor={medium-blue}
	}
\usepackage{fullpage}

\usepackage{mathbbol}

\theoremstyle{plain}
\newtheorem{thm}{Theorem}[section]
\newtheorem*{thm*}{Theorem}
\newtheorem{thmalpha}{Theorem}

\newtheorem{cor}[thm]{Corollary}
\newtheorem{coralpha}[thmalpha]{Corollary}
\newtheorem*{cor*}{Corollary}

\newtheorem*{conj*}{Conjecture}
\newtheorem{lem}[thm]{Lemma}

\newtheorem*{quest*}{Question}
\newtheorem*{claim*}{Claim}

\theoremstyle{definition}
\newtheorem{defn}[thm]{Definition}

\newtheorem{nota}[thm]{Notation}

\newtheorem*{ex*}{Example}
\newtheorem{sub-ex}[thm]{Sub-Example}
\newtheorem{counter-ex}[thm]{Counter-Example}

\newtheorem*{rem*}{Remark}

\DeclareMathOperator{\coev}{coev}

\DeclareMathOperator{\ev}{ev}

\DeclareMathOperator{\op}{op}

\DeclareMathOperator{\id}{id}
\DeclareMathOperator{\Isom}{Isom}

\DeclareMathOperator{\Irr}{Irr}

\DeclareMathOperator{\Spec}{Spec}

\newcommand{\comment}[1]{}

\newcommand{\noshow}[1]{}
\newcommand{\MR}[1]{}

\newcommand{\End}{{\sf End}}

\newcommand{\Rep}{{\sf Rep}}

\renewcommand{\Vec}{{\sf Vec}}

\newcommand{\Hilb}{{\sf Hilb}}
\newcommand{\fdHilb}{{\sf Hilb_{fd}}}
\newcommand{\rCorr}{{\mathsf{C^{*}Alg}}}

\newcommand{\PQN}{{\sf PQN}}

\newcommand{\ca}{{\alpha}}
\newcommand{\cb}{{\beta}}

\def\semicolon{;}
\def\applytolist#1{
    \expandafter\def\csname multi#1\endcsname##1{
        \def\multiack{##1}\ifx\multiack\semicolon
            \def\next{\relax}
        \else
            \csname #1\endcsname{##1}
            \def\next{\csname multi#1\endcsname}
        \fi
        \next}
    \csname multi#1\endcsname}

\def\calc#1{\expandafter\def\csname c#1\endcsname{{\mathcal #1}}}
\applytolist{calc}QWERTYUIOPLKJHGFDSAZXCVBNM;

\def\Ic#1{\expandafter\def\csname I#1\endcsname{{\mathbb #1}}}
\applytolist{Ic}QWERTYUIOPLKJHGFDSAZXCVBNM;

\def\bfc#1{\expandafter\def\csname bf#1\endcsname{{\mathbf #1}}}
\applytolist{bfc}QWERTYUIOPLKJHGFDSAZXCVBNM;

\def\sfc#1{\expandafter\def\csname s#1\endcsname{{\sf #1}}}
\applytolist{sfc}QWERTYUIOPLKJHGFDSAZXCVBNM;

\def\fc#1{\expandafter\def\csname f#1\endcsname{{\mathfrak #1}}}
\applytolist{fc}QWERTYUIOPLKJHGFDSAZXCVBNM;

\usepackage{tikz}
\usepackage{tikz-cd}
\usetikzlibrary{arrows,backgrounds,patterns.meta}
\usetikzlibrary{positioning,shadings,cd}
\usetikzlibrary{shapes}
\usetikzlibrary{backgrounds}
\usetikzlibrary{decorations,decorations.pathreplacing,decorations.markings}
\usetikzlibrary{fit,calc,through}
\usetikzlibrary{external}
\usetikzlibrary{arrows}
\tikzset{vertex/.style = {shape=circle,draw,fill=black,inner sep=0pt,minimum size=5pt}}
\tikzset{edge/.style = {->,> = latex', bend right}}
\tikzset{
	super thick/.style={line width=3pt}
}
\tikzset{
    quadruple/.style args={[#1] in [#2] in [#3] in [#4]}{
        #1,preaction={preaction={preaction={draw,#4},draw,#3}, draw,#2}
    }
}
\tikzstyle{shaded}=[fill=red!10!blue!20!gray!30!white]
\tikzstyle{unshaded}=[fill=white]
\tikzstyle{empty box}=[circle, draw, thick, fill=white, opaque, inner sep=2mm]
\tikzstyle{annular}=[scale=.7, inner sep=1mm, baseline]
\tikzstyle{rectangular}=[scale=.75, inner sep=1mm, baseline=-.1cm]
\tikzstyle{mid>}=[decoration={markings, mark=at position 0.5 with {\arrow{>}}}, postaction={decorate}]
\tikzstyle{mid<}=[decoration={markings, mark=at position 0.5 with {\arrow{<}}}, postaction={decorate}]
\tikzstyle{over}=[double, draw=white, super thick, double=]

\tikzcdset{scale cd/.style={every label/.append style={scale=#1},
    cells={nodes={scale=#1}}}}

\let\OLDthebibliography\thebibliography
\renewcommand\thebibliography[1]{
  \OLDthebibliography{#1}
  \setlength{\parskip}{0pt}
  \setlength{\itemsep}{0pt plus 0.3ex}
}

\newcommand{\acts}{\curvearrowright}

\newcommand{\tl}{\triangleleft}
\newcommand{\tr}{\triangleright}
\newcommand{\Cs}{C$^\ast$}

\newcommand{\fiBim}{{\sf Bim_{f}}}
\newcommand{\fiEnd}{{\sf End_{f}}}
\newcommand{\Sect}{{\sf Sect}}
\newcommand{\CPT}{\mathsf{End}^{\rm cpt}}
\newcommand{\FinRan}{\mathsf{FinRan}}

\renewcommand{\theenumi}{\arabic{enumi}}

\newtheorem*{lem*}{Lemma}

\theoremstyle{definition}
\newtheorem{no}[thm]{Notation}
\newtheorem*{qu*}{Questions}
\newtheorem{exa}[thm]{Example}
\newtheorem{rmk}[thm]{Remark}

\begin{document}

\title[Properly Outer Actions of Tensor Categories on \Cs-algebras]{Properly Outer Actions of Tensor Categories on \Cs-algebras}


\author[1]{\fnm{Roberto} \sur{Hern\'{a}ndez Palomares}}

\author[2]{\fnm{Miho} \sur{Mukohara}}


\affil[1]{\orgdiv{Pure Mathematics}, \orgname{University of Waterloo}, \orgaddress{\city{Waterloo}, \country{Canada}},
robertohp.math@gmail.com}

\affil[2]{\orgdiv{Department of Mathematics}, \orgname{Kyushu University}, \orgaddress{\city{Fukuoka}, \country{Japan}}, mukohara@math.kyushu-u.ac.jp}



\abstract{We discuss proper outerness for finite-index endomorphisms and finite-index bimodules of simple \Cs-algebras.
The notion of properly outer automorphisms of von Neumann algebras was introduced by Connes,
and the \Cs-algebraic analogue for automorphisms has also been extensively investigated. Proper outerness is a strengthened form of outerness for automorphisms, endomorphisms, and bimodules of \Cs-algebras, 
and may be viewed as a noncommutative analogue of topological freeness for homeomorphisms of locally compact spaces.
It is a classical theorem of Kishimoto that every outer automorphism of a simple \Cs-algebra is properly outer. As an application, he proved the simplicity of crossed product \Cs-algebras arising from outer actions of discrete groups on simple \Cs-algebras. More recently, Izumi proved that finite-index outer endomorphisms of purely infinite simple \Cs-algebras are automatically properly outer.
In this article, we extend Izumi’s result beyond the purely infinite setting. Our main result is that every finite-index outer endomorphism of a simple \Cs-algebra is properly outer. 
Understanding finite-index endomorphisms is useful for studying symmetries beyond group actions, 
including actions by unitary tensor categories.
As a consequence of our main result,
freeness for outer actions of unitary tensor categories on simple \Cs-algebras is also shown to hold automatically.
As applications, we obtain structural results about 
irreducible discrete inclusions of \Cs-algebras, potentially having infinite index,
such as \Cs-irreducibility in the sense of R{\o}rdam. }

\maketitle

\section{Introduction}
The notion of properly outer automorphisms of von Neumann algebras was introduced by Connes in \cite{MR394228} to classify automorphisms of the hyperfinite $\rm{I\hspace{-.01em}I}_{1}$ factor, and it played an important role in the classification of group actions on injective factors.
Afterwards, an analogous version for automorphisms of \Cs-algebras was introduced in \cite{MR574038, MR634163} and \cite[Section 6]{MR650187}.
Remarkably, in \cite{MR634163} Kishimoto showed that outer automorphisms of simple \Cs-algebras are automatically properly outer, implying that reduced crossed products by outer actions of discrete groups on simple \Cs-algebras are themselves simple.

Recently, there has been a growing interest in studying actions of unitary tensor categories (UTCs) on \Cs-algebras \cite{MR1900138, MR4139893, MR4328058, MR4566007, 2023arXiv231018125G, 2024arXiv240114238E, MR4717816, 2024arXiv240918161H, 2024arXiv240518429K, MR4916153, doi:10.1142/S1793525325500256}, which is a generalization of the \Cs-dynamical systems governed by discrete groups and the duals of compact quantum groups. 
In this manuscript, we extend the notion of proper outerness to the setting of UTCs acting on separable simple \Cs-algebras, and show that it holds automatically under mild assumptions. 
Actions of UTCs on \Cs-algebras are given by finite-index bimodules and endomorphisms, and the purpose of this paper is to study these when the \Cs-algebras are simple. 
Our main result is the following:
\begin{thmalpha}[Theorem \ref{thm:end}]\label{thmalpha:end}
Let $A$ be a separable simple \Cs-algebra and $\rho\colon A\to A$ be a finite-index 
outer endomorphism (i.e., $\{T\in M(A)|\ \rho(a)T = Ta\ \forall a\in A\}= 0$).
Then, $\rho$ is properly outer.
\end{thmalpha}
\noindent 
Proper outerness of automorphisms and finite-index endomorphisms has been discussed in \cite{MR634163,MR1234394, MR1900138}. 
In particular, Izumi recently proved Theorem \ref{thmalpha:end} for purely infinite simple \Cs-algebras in \cite{doi:10.1142/S0129055X24610026}. 
Our Theorem \ref{thmalpha:end}, therefore, constitutes a technical extension of these.

As a first application of Theorem \ref{thmalpha:end}, we discuss \emph{freeness} for UTC actions on simple \Cs-algebras. 
Freeness for a UTC action on a \Cs-algebra was introduced in \cite[\S4]{2024arXiv240918161H} based off certain averaging properties for outer actions of discrete groups on unital \Cs-algebras \cite{MR4010423}. 
Using freeness, it was shown in \cite[Theorem B]{2024arXiv240918161H} that reduced crossed products over connected \Cs-algebra objects by free UTC-actions on simple \Cs-algebras remain simple. 
Using Theorem \ref{thmalpha:end}, we conclude that freeness holds automatically from outerness.
\begin{coralpha}[Corollary \ref{cor:free}]\label{corbeta:free}
    Let $B$ be a unital separable simple \Cs-algebra and $\cC$ be a unitary tensor category. 
    Then every outer action $\cC\acts B$ is free.
\end{coralpha}

We now turn our attention to \emph{irreducible \Cs-discrete inclusions}. 
An inclusion of \Cs-algebras $B\subset A$ is irreducible whenever $B'\cap M(A) \cong \IC$, where $M(A)$ denotes the multiplier algebra of $A$. 
Discrete inclusions of von Neumann algebras have been studied in \cite{MR1622812,MR2561199,MR3948170},
and a version for \Cs-inclusions was introduced in \cite{doi:10.1142/S1793525325500256}.
A nondegenerate inclusion $B\overset{E}{\subset} A$ (i.e., $\overline{AB} = A = \overline{BA}$) with a faithful conditional expectation $E:A\to B$ is called \Cs-discrete if and only if $A$ has a dense subalgebra spanned by certain $B$-$B$ bimodules. 
Generic examples of \Cs-discrete inclusions are given by reduced crossed products by discrete group and compact group actions on \Cs-algebras.
As shown in \cite[Theorem 1.1]{doi:10.1142/S1793525325500256}, unital irreducible \Cs-discrete inclusions are characterized in terms of UTC-actions and \emph{\Cs-algebra objects}. 
In combination with Theorem \ref{thmalpha:end}, 
we obtain the following structural properties for irreducible \Cs-discrete inclusions.
\begin{coralpha}[Corollary \ref{cor:csirr}]\label{coralpha:csirr}
   Let $A$ and $B$ be \Cs-algebras, and 
    let $B\overset{E}{\subset}A$ be an irreducible \Cs-discrete inclusion.
    If $B$ is separable and simple,
    then the following hold.
    \begin{enumerate}
        \item When $A$ is separable,
        there is a pure state $\omega\colon A\to \IC$ with $\omega=\omega\circ E$.
        \item Every intermediate \Cs-algebra is simple. In particular, $A$ is automatically simple. 
        \item If $B$ is purely infinite,
        then every intermediate \Cs-algebra is purely infinite.
        \item $E$ is the unique conditional expectation from $A$ onto $B$.
        \item If $E$ is of finite index,
        then $B\subset A$ has the property (BEK) in \cite[Definition 2.4]{doi:10.1142/S0129055X24610026} sense.
    \end{enumerate}  
\end{coralpha}
The property (BEK) was introduced in \cite{MR1234394} in the context of \emph{quasi-product actions}. 
These are compact group actions whose dual endomorphisms are properly outer. 
Therein, it is shown that an inclusion $A^{G}\subset A$ arising from a quasi-product action of a compact group $G$ on a \Cs-algebra $A$ satisfies the conclusions of  Corollary \ref{coralpha:csirr} Items (1) and (5).
Further motivation for Corollary \ref{coralpha:csirr} comes from Izumi's recent results \cite{doi:10.1142/S0129055X24610026}, showing that minimal actions of compact groups with simple fixed point algebras are quasi-product actions. 
Moreover, Corollary \ref{coralpha:csirr} Item (2) implies that \Cs-irreducibility in the sense of R{\o}rdam \cite{MR4599249}—that all intermediate \Cs-algebras $B\subseteq D\subseteq A$ are simple—holds automatically in this context. 
\medskip

We now briefly discuss the broader context for our results involving \Cs-dynamical systems and their classification. 
In the study of group actions on \Cs-algebras, strong forms of outerness conditions, such as the Rokhlin property and isometric shift-absorption, 
play an important role (c.f. \cite{MR1344136, MR2414327, MR2644109}).
The notion of an isometrically shift-absorbing action was introduced by Gabe and Sz{\'a}bo for the dynamical Kirchberg–Phillips theorem \cite{MR4747811}.
It is known that isometric shift-absorption is equivalent to outerness for actions of discrete groups on Kirchberg algebras. 
To obtain this equivalence, Kishimoto’s results \cite{MR634163} on properly outer automorphisms are crucial. 
Moreover, Izumi characterized isometrically shift-absorbing actions of compact groups on Kirchberg algebras. 
This characterization result is deeply intertwined with the proper outerness of dual endomorphisms. 
In the context of UTC-actions on \Cs-algebras, it is expected that one needs an appropriate strong form of outerness for the associated bimodules and endomorphisms. 
Theorem \ref{thmalpha:end} shows that Izumi’s aforementioned results on proper outerness are not phenomena unique to purely infinite \Cs-algebras, and we expect it to become a useful tool for future studies of generalized \Cs-dynamical systems. 

\bigskip 

{\bf Acknowledgments:} The authors are grateful to Yuki Arano, Michael Brannan, Yasuyuki Kawahigashi, Matthew Kennedy, and Kan Kitamura for helpful discussions, support, and suggestions that improved this manuscript. 
Special thanks to Masaki Izumi for proposing a problem on proper outerness and for answering our questions in the Appendix Section. 
The authors are grateful to the anonymous referee for a careful reading of the manuscript and for many helpful comments that improved the exposition.
RHP was supported by an NSERC Discovery Grant. 
MM was supported by JSPS KAKENHI Grant Number JP26K16999.

\subsection*{Statements and Declarations}
{\bf Competing Interests:}
The authors declare that they have no competing interests.

\noindent{\bf Data availability:}
Data sharing is not applicable to this article as no new data were created or analyzed in this study.

\tableofcontents

\section{Preliminaries}\label{sec:pre}
Throughout this paper,
we use the following notations and assumptions.
\begin{itemize}
    \item We assume every \Cs-algebra in this article is $\sigma$-unital.
    \item The \Cs-algebra of compact operators on a separable Hilbert space is denoted by $\IK$.
    \item The multiplier algebra of $A$ is denoted by $M(A)$.
    The unit of $M(A)$ is denoted by $1_{A}$.
\item For any $\epsilon>0$ and any elements $a$, 
$b$ of a normed space $X$,
we write $a\approx_{\epsilon}b$ for $\|a-b\|<\epsilon$.
\item We say that the inclusion $B\subset A$ of \Cs-algebras is irreducible if $M(A)\cap B^{\prime}=\IC1_{A}$ holds.
\item For a \Cs-algebra $A$ and $a\in A$,
the spectrum $a$ is denoted by ${\rm Spec}(a)$.
\end{itemize}

We summarize basic definitions and notations on Hilbert \Cs-modules and \Cs-correspondences. 
We refer to \cite{MR1325694} for details.
\begin{defn}
Let $B$ be a \Cs-algebra.
     A {\bf right Hilbert $B$-module }$(X_{B},\langle\cdot|\cdot\rangle_{B})$ is a right $B$-module $X_{B}$ equipped with a right sesquilinear form $\langle\cdot|\cdot\rangle_{B}\colon X_{B}\times X_{B}\to B$ satisfying the following:
        \begin{enumerate}
            \item For all $\xi,\eta\in X_{B}$ and $b\in B$,
            \[\langle\xi|\eta\tl b\rangle_{B}=\langle\xi|\eta\rangle_{B}b,\;\langle\xi|\eta\rangle_{B}^{*}=\langle\eta|\xi\rangle_{B},\;\text{ and }\;\langle\xi|\xi\rangle_{B}\geq0\;\text{ hold.}\] 
            \item The norm defined by $\|\xi\|:=\|\langle\xi|\xi\rangle_{B}\|^{\frac{1}{2}}$ is a complete norm of $X_{B}$.
        \end{enumerate}
        If no confusion arises, we use the notation $X$ instead of $X_{B}$.
        We define a left Hilbert $B$-module $({}_{B}X,{}_{B}\langle\cdot,\cdot\rangle)$ in a similar way.
\end{defn}

\begin{defn}
 A map $T\colon X_{B}\to Y_{B}$ between right Hilbert $B$-modules is called {\bf adjointable} if there exists a map $T^{*}\colon Y_{B}\to X_{B}$ satisfying 
 \[
 \langle T\xi|\eta\rangle_{B}^{Y_{B}}=\langle\xi|T^{*}\eta\rangle_{B}^{X_{B}}\;\text{ for all }\;\xi\in X_{B}\; \text{ and }\;\eta\in Y_{B}.
 \]

For vectors $\xi,\eta\in X_{B}$, the following adjointable operator is denoted by $|\xi\rangle\langle\eta|$: 
\[X_{B}\ni\zeta\mapsto\xi\tl\langle\eta|\zeta\rangle_{B}\in X_{B}.\]

We put 
\begin{align*}
 \rCorr(X_B\to Y_B)&:=\{T\colon X\to Y\mid T\;\text{ is adjointable}\},\\
    \End(X_B)&:=\rCorr(X_{B}\to X_{B}),\\
    \FinRan(X_B)&:=\mathrm{span}\{|\xi\rangle\langle\eta|\colon X_{B}\to X_{B}\mid\xi,\eta\in X\}\subset\End(X_B),\\
    \text{and }\;\CPT(X_B)&:=\overline{\FinRan(X_B)}\subset\End(X_B).
\end{align*}
\end{defn}
\noindent We often refer to $\FinRan(X_B)$ (resp. $\CPT(X_B)$) as the set of {\bf finite-rank operators} (resp. {\bf compact operators}).

\begin{defn}\label{defn:correspondence}
    A \textbf{right $A$-$B$ correspondence} $_{A}X_{B}$ consists of a right Hilbert $B$-module $X_{B}$ and a $*$-homomorphism $\phi\colon A\to\End(X_{B})$.
    We put for $a\in A$, $a\rhd-:=\phi(a)\in \End(X_B)$ and for $\xi\in {}_AX_B$ we write $a\rhd\xi:=\phi(a)\xi$. 
    If no confusion arises,
    we use the notation $X$ instead of $_{A}X_{B}$.
    We define a left $A$-$B$ correspondence similarly. 

For any right $A$-$B$ correspondence $_{A}X_{B}$, we define its \textbf{contragredient correspondence} as 
$_{B}X^c_{A}=\left\{\overline{\xi}\mid \xi\in\  {}_{A}X_{B}\right\}$.
The vector space structure and the left $B$-$A$ correspondence structure of $_{B}X^c_{A}$ is given as follows:
\[z\cdot\overline{\xi}+w\cdot\overline{\eta}:=\overline{\overline{z}\xi+\overline{w}\eta},\quad b\rhd\overline{\xi}\lhd a := \overline{a^*\rhd{\xi}\lhd b^*}, \quad {}_B\langle\overline{\eta}, \overline{\xi} \rangle := \langle \eta| \xi\rangle_B\]
for $a\in A$, $b\in B$, $\xi,\eta\in X$ and $z,w\in\IC$.
There is an anti-linear conjugation isomorphism 
$$
J:X\to X^c,\ \text{mapping}\ \xi\mapsto \overline{\xi}. 
$$
\end{defn}
 For any right $A$-$B$ correspondences $X$, $Y$, we put 
 \begin{align}
    \rCorr({}_AX_B\to {}_AY_B)&:=\{T\in\rCorr(X_{B}\to Y_{B})\mid T\circ a\rhd=(a\rhd) \circ T\;\text{ for all }a\in A\},\label{eqn:BilMaps}\\
    \End({}_AX_B)&:=\rCorr({}_AX_B\to {}_AX_B).\nonumber 
 \end{align}
  We recall the basics of tensor products of \Cs-correspondences (see \cite[Chapter 4]{MR1325694} for details).
   Let $_{A}X_{B}$ and $_{B}Y_{C}$ be \Cs-correspondences. 
   The algebraic relative/interior tensor product of bimodules $_{A}X_{B}$ and $_{B}Y_{C}$ is denoted by $_{A}X\boxdot_{B}Y_{C}$ and is given by the \emph{separation} 
   $$
   X\odot_\IC Y/{\rm span}_\IC\left\{ \xi\lhd b\otimes\eta - \xi\otimes b\rhd\eta\right\}_{\xi\in X, \eta\in Y, b\in B},
   $$
   with the obvious $A$-$C$ bimodule structure.
   The {\bf relative/interior tensor product} ${}_A(X\boxtimes_B Y)_C$ is the $A$-$C$ correspondence given by the completion of $X\boxdot_B Y$ with respect to the inner product
   $$
    \langle\xi_{1}\boxdot\eta_{1}|\xi_{2}\boxdot\eta_{2}\rangle^{X\boxtimes_{B} Y}:=\langle\eta_{1}|\langle\xi_{1}|\xi_{2}\rangle^{X}\rhd\eta_{2}\rangle^{Y}.
   $$

    The {\bf exterior tensor product} ${}_{A\otimes_{\min} C}(X\otimes Y)_{B\otimes_{\min}D}$ is a $A\otimes_{\min} C$-$B\otimes_{\min} D$ correspondence,
    which is a completion of ${}_{A}X_{B}\odot_{\IC} {}_{C}Y_{D}$ with respect to the following inner product:
    \[\langle\xi_{1}\otimes\eta_{1}|\xi_{2}\otimes\eta_{2}\rangle^{X\otimes Y}:=\langle\xi_{1}|\xi_{2}\rangle^{X}\otimes\langle\eta_{1}|\eta_{2}\rangle^{Y}.\]
    
    \begin{exa}\label{exa:basic_construction}
        Let $B\overset{E}{\subset}A$  be a nondegenerate inclusion (i.e. $\overline{AB}=\overline{BA}=A$) of \Cs-algebras with a faithful conditional expectation $E$.
        The right $A$-$B$ correspondence $\cE=(\cE, \langle\cdot|\cdot\rangle_{B})$ associated with $E$ is the completion of $A$ with respect to the norm $\|a\|_{E}:=\|E(a^{*}a)\|^{\frac{1}{2}}$.
        That is,
        there exists an injective bounded $A$-$B$ bilinear map with dense image
        \[
        \eta\colon A\to \cE \qquad\text{ satisfying }\qquad \langle\eta(a)|\eta(b)\rangle_{B}=E(a^{*}b)\quad \forall a,b\in A.
        \]
        The {\bf Jones projection} is a projection $e\in\End({}_{B}\cE_{B})$ defined as \[e\eta(a):=\eta(E(a))\] for every $a\in A$.
        Using these notations,
        for each $a,b\in A$,
        we have \[|\eta(a)\rangle\langle\eta(b)|=aeb^{*}\in\End(\cE_{B}).\]
        In this case, the \Cs-algebra 
        $$
        A_1:=\CPT(\cE_{B})=\overline{\rm span}AeA
        $$ is called  {\bf the basic construction}.
        We see that the multiplier algebra satisfies $M(A_{1})=\End(\cE_{B})$.
    \end{exa}

\subsection{\texorpdfstring{Finite-index inclusions of simple \Cs-algebras}{}}\label{sec:fin}
The Watatani index for unital inclusions of \Cs-algebras was introduced by Watatani in \cite{MR996807}, generalizing the Jones index for subfactors \cite{MR696688}. 
We recall the generalized definition based on \cite{MR1900138}.
Let $B\overset{E}{\subset}A$ be a non-degenerate inclusion of \Cs-algebras with a conditional expectation $E$ from $A$ onto $B$.
We use the notation introduced in Example \ref{exa:basic_construction}.
\begin{defn}[\cite{MR860811}]
The {\bf Pimsner-Popa index} $\mathrm{Ind}_{\mathrm{p}}\:E$ of $E$ is defined as follows:
\[
\mathrm{Ind}_{\mathrm{p}}\:E:=\inf\{\lambda>0\mid\lambda E-\id_{A}\text{ is completely positive.}\}.
\]
\end{defn}
\noindent Observe that if $\mathrm{Ind}_{\mathrm{p}}\:E<\infty$, then $E$ is automatically faithful.
\begin{defn}[Theorem 2.8 of \cite{MR1900138}]\label{def_Watatani}
If $\mathrm{Ind}_{\mathrm{p}}\:E<\infty$ and $A\subset A_{1}$,
then there is a bounded completely  positive $A$-$A$-bimodule map $\hat{E}$ from $M(A_{1})$ onto $M(A)$ satisfying $\hat{E}(e)=1$ and $\hat{E}(1_{M(A_1)})\in Z(M(A))$.
The {\bf Watatani index} of $E$ is defined by 
\begin{equation*}
\mathrm{Ind}_{\mathrm{w}}\: E:=
\begin{cases}
\hat{E}(1_{M(A_1)})&\text{if $A\subset A_1$, }\\
\infty&\text{otherwise.}
\end{cases}
\end{equation*}
\end{defn}

The next result due to Izumi guarantees that $B\subset A\subset A_1$ when $A$ and $B$ are simple. 
\begin{thm}[Corollary 3.7 of \cite{MR1900138}]
If $B\overset{E}{\subset}A$ is a non-degenerate inclusion of simple \Cs-algebras with $\mathrm{Ind}_{\mathrm{p}}\: E<\infty$,
then we have $A\subset A_1$ and $\mathrm{Ind}_{\mathrm{p}}\:E=\mathrm{Ind}_{\mathrm{w}}\:E$.
\end{thm}
\noindent As a consequence, we write $\mathrm{Ind}\:E$ instead of $\mathrm{Ind}_{\mathrm{w}}\:E$  whenever $A$ and $B$ are simple. 

\smallskip
Next, we briefly summarize relevant aspects of finite-index endomorphisms and the sector theory of a \Cs-algebra that we shall need.
(See Section 4 of \cite{MR1900138} for details.)
\begin{defn}\label{defn:finite_index_endomorphism}
    Let $A$ be a \Cs-algebra.
    \begin{enumerate}
        \item A $*$-endomorphism $\rho\colon A\to A$ is said to have a {\bf finite index} if there exists a finite-index conditional expectation $E\colon A\rightarrow\rho(A)$.
When $\rho$ is of finite index, there is a unique conditional expectation 
$$
E_{\rho}\colon A\rightarrow\rho(A)
$$
which has a {\bf minimal index} (see Theorem 2.12.3 of \cite{MR996807} and Section 3 of \cite{MR1900138}). The {\bf statistical dimension of $\rho$} is defined as the square root $\sqrt{\mathrm{Ind}_{\mathrm{w}}\: E_{\rho}}$ of the minimal index and is denoted by $d(\rho)$.
\item A $*$-endomorphism $\rho\colon A\to A$ is called irreducible if $\rho(A)\subset A$ is irreducible.
    \end{enumerate}
\end{defn}
\begin{no}\label{notation:sectors}
Let $A$ be a stable simple \Cs-algebra.
\begin{itemize}
\item The set of finite-index $*$-endomorphisms of $A$ is denoted by $\End_{\rm f}(A)$.
\item For any $*$-homomorphisms $\rho_{1},\rho_{2}\colon A\to B$, we put
\[(\rho_{1}, \rho_{2}):=\{T\in M(B)\mid T\rho_{2}(a)=\rho_{1}(a)T,\mathrm{\ for\ all\ } a\in A\}.\]
\item For endomorphisms $\rho_{1},\rho_{2}\colon A\to A$, define the composition $\rho_{1}\circ\rho_{2}$ by $\rho_{1}\circ\rho_{2}(a):=\rho_{1}(\rho_{2}(a))$ for all $a\in A$.
\item  Two endomorphisms $\rho_{1},\rho_{2}\in\fiEnd(A)$ are said to be unitary equivalent and denoted by $\rho_{1}\sim\rho_{2}$,
if there is a unitary $U\in(\rho_{1}, \rho_{2})$.
\item the unitary equivalence class $[\rho]$ of $\rho\in\fiEnd(A)$ is called a {\bf sector}.
We say that $[\rho]$ is irreducible if $\rho$ is irreducible.
\item The set of sectors $\fiEnd(A)/\sim$ is denoted by $\Sect(A)$.
\end{itemize}
\end{no}
Since $A$ is stable,
there are isometries $S_{1}, S_{2}\in M(A)$ with $S_{1}S_{1}^{*}+S_{2}S_{2}^{*}=1_{M(A)}$.
We can define the {\bf product and the sum of $\Sect(A)$} as 
\[
[\rho_{1}][\rho_{2}]:=[\rho_{2}\circ\rho_{1}]\qquad\text{and}\qquad [\rho_{1}]\oplus[\rho_{2}]:=[\rho],
\]
where $\rho$ is an endomorphism of $A$ such that
\[
\rho(a)=S_{1}\rho_{1}(a)S_{1}^{*}+S_{2}\rho_{2}(a)S_{2}^{*}
\] for all $a\in A$.
That is,
for any $\rho_{1},\rho_{2}\in\fiEnd(A)$,
$\rho_{1}\circ\rho_{2}$ and $\rho$ defined as above are also of finite index, and $[\rho]$ does not depend on the choice of $S_{1},S_{2}\in M(A)$.

The following statements involve an important property and structure of sectors—resembling semi-simplicity and dualizability in the context of tensor categories—and we record them here for the reader's convenience.
\begin{lem}[Lemma 4.1 of \cite{MR1900138}]\label{lem:semisimple_sectors}
   When $A$ is a separable simple stable \Cs-algebra,
   every $[\rho]\in\Sect(A)$ is uniquely decomposed into a direct sum of irreducible sectors.
\end{lem}
\begin{lem}[Lemma 4.4 of \cite{MR1900138}]\label{conjugate}
Let $A$ be a separable simple stable \Cs-algebra.
For any $\rho\in\fiEnd(A)$, 
there are $\overline{\rho}\in\fiEnd(A)$ with $d(\rho) = d(\overline{\rho}),$ and isometries $R_{\rho},\overline{R}_{\rho}\in M(A)$ satisfying
\begin{align*}
    R_{\rho}\in(\overline{\rho}\circ\rho,\;\id_{A}),&\qquad \overline{R}_{\rho}\in(\rho\circ\overline{\rho},\;\id_{A}),\;\text{ and }\\
    \qquad\overline{R}_{\rho}^{*}\rho(R_{\rho})=&R_{\rho}^{*}\overline{\rho}(\overline{R}_{\rho})=\frac{1}{d(\rho)}.
\end{align*} 
Moreover, 
$$
E_{\rho}(-):=\rho(R^{*}_\rho\overline{\rho}(-)R_\rho)
$$
(resp. $E_{\overline{\rho}}(-):=\overline{\rho}(\overline{R}^{*}_\rho\rho(-)\overline{R}_\rho)$) is the conditional expectation from $A$ to $\rho(A)$ (resp. $\overline{\rho}(A)$) attaining the minimal index. 
In that case, $(d(\rho) \overline{R}^{*}_\rho, d(\rho)\overline{R}_\rho)$ and $(d(\rho) R^{*}_\rho, d(\rho)R_\rho)$ are quasi-bases, respectively.
That is, we have 
\begin{align*}
    \id_{A}&=d(\rho)^{2}\overline{R}^{*}_{\rho} E_{\rho}(\overline{R}_{\rho} \,-)=d(\rho)^{2}E_{\rho}(-\,\overline{R}^{*}_{\rho})\overline{R}_{\rho},\\
    \id_{A}&=d(\rho)^{2}R^{*}_{\rho} E_{\overline{\rho}}(R_{\rho} \,-)=d(\rho)^{2}E_{\overline{\rho}}(-\,R^{*}_{\rho})R_{\rho}.    
\end{align*}

\end{lem}
We comment that the conjugate endomorphism above was constructed by Izumi as $\overline{\rho} = \rho^{-1}\circ\gamma$, where $\gamma$ is the restriction to $A$ of the canonical isomorphism $\gamma_1:A_1\to \rho[A]$ afforded by the stability of $A$.


\subsection{\texorpdfstring{Dualizable bimodules over \Cs-algebras}{}}\label{sec:finite_index_bimodules}
In this subsection,
all right \Cs-correspondences $X$ over $A$ that we consider are assumed to be left nondegenerate; that is, $\overline{A\rhd X} = X$.
We shall now summarize basic facts on \emph{dualizable} \Cs-correspondences based on \cite{MR2085108}.
\begin{defn}[Definition 4.3 of \cite{MR2085108}]\label{defn:dual_correspondence}
  Let $X$ be a right \Cs-correspondence over $A$.
  We say that $X$ has a {\bf conjugate/dual} object $\overline{X}$ if $\overline{X}$ is a right \Cs-correspondence over $A$ and there are $R_{X}\in\rCorr(_{A}A_{A}\to _{A}\overline{X}\boxtimes_{A}X_{A})$ and $\overline{R}_{X}\in\rCorr({}_{A}A_{A}\to {}_{A}X\boxtimes_{A}\overline{X}_{A})$ satisfying
  \begin{align*}
    (\overline{R}_{X}^{*}\boxtimes\id_{X})\circ(\id_{X}\boxtimes R_{X})=\id_{X}\qquad\quad\text{ and }\qquad\quad(R_{X}^{*}\boxtimes\id_{\overline{X}})\circ(\id_{\overline{X}}\boxtimes \overline{R}_{X})=\id_{\overline{X}}.  
  \end{align*}
  We call $X$ is {\bf dualizable} if a dual object and corresponding morphisms $(\overline{X}, R_X, \overline{R}_X)$ exist.
\end{defn}
\begin{rmk}\label{rmk:finite_index}
As in \cite[Theorem 4.4 and 4.13]{MR2085108},
it is known that a \Cs-correspondence is dualizable if and only if it is a {\bf finite-index bi-Hilbertian \Cs-bimodule} \cite[Definition 2.23]{MR2085108}.
A bi-Hilbertian $A$-$A$-bimodule is a right \Cs-correspondence $X$ over $A$ equipped with a compatible left \Cs-correspondence structure \cite[Definition 2.3]{MR2085108}.
If $X$ is a finite-index bi-Hilbertian \Cs-bimodule, the dual object $\overline{X}$ can be described as the contragredient bimodule $X^{c}$ in Definition \ref{defn:correspondence} \cite[\S4]{MR2085108}.
For simplicity of notation, we write {\bf finite-index \Cs-bimodule} instead of bi-Hilbertian \Cs-bimodule of finite index.
\end{rmk}

While there is a rigorous definition of index for a general \Cs-correspondence \cite{MR2085108}, due to its technical nature, here, we shall only give a rough sketch that will be sufficient for our purposes.  
Below, we phrase this in terms of the index for inclusions that we introduced above. 
\begin{rmk}\label{rmk:inclusion_vs_bimodule}
Let $X$ be a finite-index bimodule over $A$ with $M(A)\cap A^{\prime}=\IC$, 
then $X$ is \textbf{faithful} (i.e. the left and right actions of $A$ on $X$ are faithful) and \textbf{proper} (i.e. $\phi_{X}(A)\subset \CPT(X_{A})$) \cite[Theorem 2.22, Corollary 2.28]{MR2085108}.
As in \cite[Corollary 2.29]{MR2085108},
we have a conditional expectation 
\[E_{X}\colon\CPT(X_A)\to \phi_{X}(A)\] 
of finite Watatani index. (Recall that finite-index expectations are automatically faithful.)

Conversely, 
suppose $X$ is a faithful proper nondegenerate right \Cs-correspondence with left action $\phi_{X}$. 
If there is a conditional expectation $E_{X}\colon\CPT(X_A)\to \phi_{X}(A)$ of finite Pimsner-Popa index, then $X$ has a left \Cs-correspondence structure with the following left inner product:
\[_{A}\langle\xi,\eta\rangle:=\phi_{X}^{-1}(E_{X}(|\xi\rangle\langle\eta|)),\quad\text{for every }\xi,\eta\in X.\]
In this case,
if the contragredient bimodule $X^{c}$ is also proper,
then $X$ is a finite-index \Cs-bimodule.
Indeed, by \cite[Theorem 2.22]{MR2085108}, it is enough to show that $X$ and $X^{c}$ are of finite \emph{right numerical indices} in the sense of \cite[Definition 2.8]{MR2085108}.
That is, there exist $\lambda, \lambda'>0$ satisfying the following: 
\[ 
\bigg\|\sum_{i=1}^{n}{}_{A}\langle\overline{\xi}_{i},\overline{\xi}_{i}\rangle\bigg\| \leq \lambda \bigg\|\sum_{i=1}^{n}|\overline{\xi}_{i}\rangle\langle\overline{\xi}_{i}|\bigg\|,
\qquad\bigg\|\sum_{i=1}^{n}{}_{A}\langle\xi_{i},\xi_{i}\rangle\bigg\|\leq\lambda^{\prime}\bigg\|\sum_{i=1}^{n}|\xi_{i}\rangle\langle\xi_{i}|\bigg\|,\quad\forall\xi_{1},\dots,\xi_{n}\in X.\]
The second inequality follows from the definition of the left inner product on $X$ by taking $\lambda'=1$.
Since $E_{X}$ is of finite Pimsner-Popa index,
as in \cite[Theorem 1]{MR1642530},
there exists $\lambda>0$ such that $\lambda E_{X}-\id_{\CPT(X_{A})}$ is completely positive.
Hence, we get the first inequality as follows 
\begin{align*}
\bigg\|\sum_{i=1}^{n}{}_{A}\langle\overline{\xi}_{i},\overline{\xi}_{i}\rangle\bigg\|=
\bigg\|\sum_{i=1}^{n}\langle\xi_{i}|\xi_{i}\rangle_{A}\bigg\| = &\bigg\|\sum_{i,j}|\xi_{i}\rangle\langle\xi_{j}|\otimes e_{i,j}\bigg\| \\
\leq &\lambda\bigg\|\sum_{i,j}{}_{A}\langle\xi_{i},\xi_{j}\rangle\otimes e_{i,j}\bigg\|
 = \lambda\bigg\|\sum_{i,j}\langle\overline{\xi}_{i}|\overline{\xi}_{j}\rangle_{A}\otimes e_{i,j}\bigg\|
=\lambda\bigg\|\sum_{i=1}^{n}|\overline{\xi}_{i}\rangle\langle\overline{\xi}_{i}|\bigg\|,
\end{align*}
where $\{e_{i,j}\}_{i,j=1}^{n}$ is a system of matrix units of $M_{n}$.

\end{rmk}

\begin{exa}\label{exa:bimodule}
Let $A$ be a separable simple stable \Cs-algebra,
$\rho$ be a finite-index endomorphism of $A$ and $X:={}_{\rho}A$ be a right \Cs-correspondence associated with $\rho$,
that is $X$ is a trivial right Hilbert $A$-module equipped with a left action $\rho\colon A\to A\cong\CPT(X_A)$.
In this case,
using $\overline{\rho}\in\fiEnd(A)$ in Lemma \ref{conjugate},
${}_{\overline{\rho}}A$ is a dual \Cs-correspondence of $X$ and $R_{\rho}$ (resp. $\overline{R}_{\rho}$) in Lemma \ref{conjugate} corresponds to ${d(\rho)}^{-\frac{1}{2}}\overline{R}_{X}$ (resp. ${d(\rho)}^{-\frac{1}{2}}R_{X}$) in Definition \ref{defn:dual_correspondence}.
Moreover, by the definition of finite-index endomorphisms,
we see that there is a finite-index conditional expectation from $A\cong\CPT(X_A)$ onto $\rho(A)$.
\end{exa}

\begin{defn}\label{defn:fiBim}
For a \Cs-algebra $A$,
we define $\fiBim(A)$ to be the category {\bf of dualizable right \Cs-correspondences} (equivalently, finite-index bimodules) over $A$, whose space of morphisms from $X$ to $Y$ is given by $\rCorr({}_AX_A\to {}_AY_A)$ (see Equation (\ref{eqn:BilMaps})).
\end{defn}

We believe that the following results are well known to experts, but since we did not find a proof available in the literature that we could readily use, we provide proofs here for the reader's convenience.
\begin{lem}[Lemma 2.8 of \cite{doi:10.1142/S1793525325500256} for unital case]\label{lem:orthogonal_summand}
    Let $A$ be a \Cs-algebra with $M(A)\cap A^{\prime}=\IC$,
    $X$ and $Y$ be \Cs-correspondences over $A$ and $f\in\rCorr(_{A}X_{A}\to {}_{A}Y_{A})$.
    If $X$ is in $\fiBim(A)$,
    then $f(X)$ (resp. $\ker(f)$) is an orthogonal summand of $Y$ (resp. $X$) which has a finite index. 
\end{lem}
\begin{proof}
    We first show $\End({}_AX_A)$ is finite dimensional.
    Let $\phi_{X}(A)\overset{E_{X}}{\subset}\CPT(X_{A})$ be the finite-index inclusion in Remark \ref{rmk:inclusion_vs_bimodule}.
    The extension of $E_{X}$ from $\End(X_A)$ onto $M(\phi(A))$ is also denoted by $E_{X}$ (see \cite[Lemma 2.6]{MR1900138}).
    By the assumption $M(A)\cap A^{\prime}=\IC$,
    the restriction $E_{X}|_{\End({}_AX_A)}$ of $E_{X}$ on $\End({}_AX_A)$ is a state of finite index.
    By \cite[Proposition 2.7.3]{MR996807},
    $\End({}_AX_A)$ is finite dimensional.

    Since $f^{*}f\in\End({}_AX_A)$ has a finite spectrum $\sigma(f^{*}{f})$ and we have $\sigma(f^{*}{f})\setminus\{0\}=\sigma(ff^{*})\setminus\{0\}$,
    $ff^{*}$ also has a finite spectrum.
    Taking the spectral projection $p=\mathbb{1}_{\{0\}}(ff^{*})\in\End({}_{A}Y_{A})$,
    we have $ff^{*}(Y)=p^{\perp}(Y)$ and $pf=0$.
    This implies $ff^{*}(Y)\subset f(X)\subset p^{\perp}(Y)$, then we get $f(X)=p^{\perp}(Y)$.
    Similarly,
    let $q:=\mathbb{1}_{\{0\}}(f^{*}f)\in\End({}_{A}X_{A})$,
    then we have $\ker(f)=\ker(f^{*}f)=q(X)$.
    We see that $f(X)$ and $\ker(f)$ are orthogonal summands.
Since $f(X)$ is a complementable submodule, the restriction $f|_{\ker(f)^{\perp}}$ is an adjointable $A$-$A$ bilinear isomorphism between $\ker(f)^{\perp}$ and $f(X)$.

For the rest of the statement, we now show that $qX\cong\ker(f)$ and $q^{\perp}X\cong f(X)$ are finite-index bimodules.
    The left action $\phi_{qX}$ on $qX$ is given by $\phi_{X}(\cdot)q\colon A\to\CPT(qX_{A})=q\CPT(X_{A})q$.
    Hence, whenever $q\neq0$ we get the conditional expectation with finite Pimsner-Popa index as follows:
     \[E_{qX}\colon\CPT(qX_{A})\ni x\mapsto E_{X}(q)^{-1}qE_{X}(x)\in \phi_X(A)q,\]
     where we remark that $E_{X}(q)$ is a positive scalar.
As in Remark \ref{rmk:inclusion_vs_bimodule},
it is enough to show that $\phi_{qX}$ is faithful and the contragredient bimodule $(qX)^{c}$ is proper.
Since we have $\phi_{X}=E_{X}(q)^{-1}E_{X}\circ\phi_{qX}$ and $\phi_{X}$ is faithful,
$\phi_{qX}$ is also faithful.
Let $J\colon X\to X^{c}$ be the canonical antilinear isomorphism.
As in \cite[Corollary 4.6]{MR2085108},
we see that $JqJ^{-1}\in\End(_{A}X^{c}_{A})$ and $JqJ^{-1}(X^{c})=(qX)^{c}$. 
Arguing as in the first half of the proof,
it follows that $(qX)^{c}$ is an orthogonal summand of $X^{c}$.
Since $X^{c}$ is proper,
so is $(qX)^c$.
Hence,
$qX$ (resp. $q^{\perp}X$) is in $\fiBim(A)$.
\end{proof}
It is easy to check that $\fiBim(A)$ is closed under orthogonal sums and relative tensor products.
As in Lemma \ref{lem:semisimple_sectors},
we will see that $\fiBim(A)$ is semisimple if $M(A)\cap A^{\prime} \cong \mathbb{C}$ holds.
We say that $X\in\fiBim(A)$ is {\bf irreducible} if $\End({}_{A}X_{A})\cong M(A)\cap A^{\prime}$ holds.
\begin{lem}\label{lem:semisimple_bimodule}
    Let $A$ be a \Cs-algebra with $M(A)\cap A^{\prime}=\IC$,
    then every $X\in\fiBim(A)$ is decomposed into a finite direct sum of irreducible bimodules in $\fiBim(A)$.
\end{lem}
\begin{proof}
In the same way as in the proof of Lemma \ref{lem:orthogonal_summand},
we see that $\End({}_{A}X_{A})$ is finite dimensional.
Taking a finite family $\{p_{i}\}_{i=1}^{n}\subset\End({}_AX_A)$ of pairwise orthogonal minimal projections with $\sum_{i=1}^{n}p_{i}=1$.
We see that $X\cong\bigoplus_{i=1}^{n}p_{i}X$ and $\End({}_A(p_{i}X)_A)\cong\IC$ for every $i$.
Thanks to Lemma \ref{lem:orthogonal_summand},
each $p_{i}X$ is in $\fiBim(A)$,
then we get the statement.
\end{proof}

\begin{lem}[Lemma 2.8 (1) of \cite{doi:10.1142/S1793525325500256} for unital case]\label{lem:adjointable}
    Let $X$ and $Y$ be right \Cs-correspondences over $A$.
    If $X$ is in $\fiBim(A)$ and the left action of $A$ on $Y$ is nondegenerate,
    then every bounded $A$-$A$ bilinear map $g\colon X\to Y$ is adjointable.
\end{lem}
\begin{proof}
For each $x\in\CPT(X_{A})$,
we put $\tilde{g}(x):=g\circ x\colon X\to Y$.
Since $\tilde{g}(|\xi\rangle\langle\eta|)=|g(\xi)\rangle\langle\eta|$ for every $\xi,\eta\in X$,
we get $\tilde{g}(\FinRan(X_{A}))\subset \rCorr(X_{A}\to Y_{A})$.
By construction,
    $\tilde{g}$ is continuous with respect to the operator norm.
    Hence,
   $\tilde{g}$ is a bounded $A$-$A$ bilinear map \[\tilde{g}:\CPT(X_{A}) \to\rCorr(X_{A}\to Y_{A}).\]
    Let $\phi_{X}\colon A\to\End(X_{A})$ and $\phi_{Y}\colon A\to\End(Y_{A})$ be the left actions, and take an approximate unit $(x_{\mu})_{\mu}$ of $\CPT(X_{A})$. 
We now show that the bounded net $(\tilde{g}(x_{\mu}))_{\mu}\subset \rCorr(X_{A}\to Y_{A})$ converges in the strong $*$-topology.

    Since $\phi_{X}(A)$ is contained in $\CPT(X_{A})$ and $\tilde{g}$ is bounded,
    for every $a\in A$ and $\xi\in X$ (resp. $\eta\in Y$), the net
    $\tilde{g}(x_{\mu})\phi_{X}(a)\xi=\tilde{g}(x_{\mu}\phi_{X}(a))\xi$ (resp. $\tilde{g}(x_{\mu})^{*}\phi_{Y}(a)\eta=\tilde{g}(\phi_{X}(a)^{*}x_{\mu})^{*}\eta$) converges.
    As $X$ and $Y$ are left nondegenerate and the nets $(\tilde{g}(x_{\mu}))_{\mu}$ and $\tilde{g}(x_{\mu})^{*}$ are bounded,
    we can define bounded maps \begin{align*}T
    \colon &X\longrightarrow Y\qquad\qquad\text{ and }\qquad\qquad S
    \colon Y\longrightarrow X\\
    &\xi\mapsto\lim_{\mu}\tilde{g}(x_{\mu})\xi\qquad\qquad\qquad\qquad\quad\eta\mapsto\lim_{\mu}\tilde{g}(x_{\mu})^{*}\eta.
    \end{align*}
   It is straightforward to show that $T=S^{*}$ (i.e., $T$ is adjointable). By construction,
    we have $T\xi=\lim_{\mu}g\circ x_\mu(\xi)=\lim_{\mu}g(x_{\mu}\xi)=g(\xi)$ for every $\xi$,
    and $g$ is adjointable.
\end{proof}

\subsection{\texorpdfstring{Unitary tensor categories and their actions on \Cs-algebras}{}}\label{sec:UTC}
We shall briefly recall some basics of \emph{rigid \Cs-tensor categories}, along with some relevant examples. The specialized audience knows well that these can be viewed as a \Cs 2-category over a single object \cite[Remark 2.10]{MR4419534}, and can safely skim through this section. However, aiming for the non-specialized audience, we shall provide a more elementary and detailed but brief approach.
The interested reader can see \cite{MR3242743, doi:10.1142/S1793525325500256} and references therein.

All categories considered in this manuscript are assumed to be \emph{essentially small} (isomorphism classes of objects form a set) and are \emph{enriched over complex vector spaces} (Hom spaces are $\IC$-vector spaces) unless otherwise specified. 
As such, we can take direct sums of objects and morphisms, denoted by $-\oplus-$. 

We let $(\cC, \circ, \otimes, 1_\cC)$ denote a \textbf{tensor category}, where $-\circ-$ denotes the associative composition of morphisms as linear maps (we write the composition $-\circ-$ from right to left. e.g. for $f\in\cC(a\to b)$ and $g\in\cC(b\to c),$ we have $g\circ f\in\cC(a\to c)$), $-\otimes-: \cC\times \cC$ is the tensor product bilinear functor, and $1_\cC$ denotes the tensor unit. We write the tensor product of morphisms $f\in \cC(a\to a')$ and $g\in\cC(b\to b')$ as $f\otimes g\in\cC(a\otimes b\to a'\otimes b')$.
Strictly speaking, tensor categories come equipped with the data of an \emph{associator, and left/right unitors} natural isomorphisms ($\alpha_{a,b,c}: a\otimes(b\otimes c)\cong (a\otimes b)\otimes c,$ $a\otimes 1_\cC\cong a$ and $1_\cC\otimes a\cong a$ respectively, for $a,b,c\in \cC$); however, we will completely suppress this data as it shall not play a role in this manuscript, and most of the categories we will focus on are \emph{strict}. 
In fact, for ease of notation, we shall denote tensor categories simply by $\cC$. 

We say a tensor category $\cC$ is a \textbf{$*$-tensor category} if for every $a,b\in \cC$, the Hom space $\cC(a\to b)$ has an involution/dagger structure
$$
*:\cC(a\to b)\to \cC(b\to a),
$$
such that for all $\lambda\in \IC$, $a,b,c\in \cC$ and $f\in\cC(a\to b),g\in \cC(b\to c)$, we have $(f^*)^*= f,$ $(\lambda f)^* = \overline{\lambda}f^*$ is conjugate-linear, $(g\circ f)^* = f^*\circ g^*$ and $(f\otimes g)^* = f^*\otimes g^*.$

\begin{defn}\label{defn:CstarCat}
    A $*$-tensor category is a \textbf{\Cs-tensor category} if it satisfies the following conditions:
    \begin{enumerate}[label=(C*\arabic*)]
        \item For each pair of objects $a,b\in \cC,$ there is a conjugate-linear involution $*:\cC(a\to b)\to\cC(b\to a)$ such that for each pair of composable morphisms $f$ and $g$ we have $(f\circ g)^*=g^*\circ f^*.$
        \item There is a \Cs-norm on $\cC(a\to b);$ that is, the identities $||f\circ f^*||=||f^*\circ f||= ||f||^2$ hold on $\cC(a\to b)$.
        \item For each $f\in\cC(a\to b),$ there exists $g\in\cC(a\to a)$ such that $f^*\circ f = g^*\circ g.$
    \end{enumerate}
    A $*$-functor $F:\cC\to \cD$ between \Cs-categories satisfies $F(f^*)=F(f)^*$ for every morphism $f$ in $\cC.$
\end{defn}

We will assume that all of our \Cs-tensor categories are unitarily \emph{Cauchy complete}; this means that $\cC$ admits all orthogonal direct sums, and also that all idempotents split via an isometry.
For further details see \cite[Assumption 2.7]{MR4419534}.

We say a tensor category $\cC$ is \textbf{rigid/has duals} if for every object  $c\in \cC$ there exists some  $\overline{c}\in\cC$ together with \textbf{evaluation and coevaluation maps} $\ev_{c}\in\cC(\overline{c}\otimes c\Rightarrow 1_\cC)$ and $\coev_c\in\cC(1_\cC\Rightarrow c\otimes \overline{c})$ 
satisfying the \emph{Zig-Zag/Duality/Snake equations}: $\id_c=(\id_c\otimes \ev_c)\circ(\coev_c\otimes \id_c)$ and $\id_{\overline c} = (\ev_c\otimes \id_{\overline c})\circ(\id_{\overline c}\otimes\coev_c).$ 
Additionally, we assume that each object $c\in\cC$ has a predual object $c_\vee\in\cC$ with $\overline{c_\vee}\cong c$ in $\cC.$ 
Throughout this manuscript, whenever $\cC$ is a rigid \Cs-tensor category, we will assume we have fixed a particular  \emph{unitary dual functor} $\overline{\ \cdot\ }$ on $\cC$ so that $(\overline{f})^*=\overline{f^*}$ for all morphisms in $\cC.$ 
Furthermore, the choice of unitary dual functor can be made \emph{balanced} whenever the corresponding \Cs-tensor category has a simple unit. 
The term balanced means that the left and right categorical traces induced by the chosen dual match. 
We refer the interested reader for more details on unitary dual functors on \Cs-tensor categories to  \cite[Definition 2.9]{MR4419534} and \cite{MR4133163}. 

Now we introduce the categories we are most interested in in this manuscript:  
\begin{defn}
A \textbf{unitary tensor category} (\textbf{UTC}) $\cC$ is a  rigid $\rm C^*$-tensor category with an irreducible unit object $1_{\cC}$. That is, $\End_\cC(1_\cC)\cong\IC$. 
If the isomorphism classes of simple objects in a unitary tensor category form a finite set, we say the category is a \textbf{unitary fusion category}. 
\end{defn}
\noindent We remind the reader that UTCs are automatically semisimple, since the simplicity of the unit along with rigidity implies that all Hom spaces are finite dimensional. We typically choose a complete set of representatives of isomorphism classes of irreducible objects, denoted $\Irr(\cC).$
We use $\cC^{\op}$ to denote the tensor category $\cC$ with the reversed tensor product structure.

\begin{exa}
    Given a separable \Cs-algebra $A$, we denote by $\fiBim(A)$ the rigid \Cs-tensor category of \textbf{dualizable $A$-$A$ bimodules} along with bounded $A$-$A$ bilinear maps. (cf Definition \ref{defn:fiBim}) 
    By Lemma \ref{lem:adjointable}, these maps are automatically right-adjointable, turning endomorphism spaces into \Cs-algebras.
    The unit object corresponds to the trivial bimodule ${}_AA_A$, and the direct sum $\oplus$ is the usual orthogonal direct sum of bimodules and intertwiners.
    The tensor product of $X,Y\in\fiBim(A)$ is the \textbf{relative tensor product/Connes' fusion} $X\boxtimes_AY$ discussed in Section \ref{sec:finite_index_bimodules}. The tensor product of intertwiners corresponds to the relative tensor product of the underlying $A$-$A$ bimodular maps. 

    If $A'\cap M(A)\cong \IC$, then the unit ${}_AA_A\in \fiBim(A)$ is irreducible (i.e. $\End({}_AA_A)\cong A'\cap M(A)\cong \IC$) and so $\fiBim(A)$ is a UTC.
\end{exa}

We believe that the following results are well known to experts. We provide proofs here for the reader's convenience. A similar argument is used in \cite[Section 3]{MR1900138}.
\begin{lem}\label{lem:stable_bimodule}
Let $A$ be a separable simple \Cs-algebra.
If $A$ is stable or unital purely infinite in the Cuntz standard form,
then every finite-index bimodule $X$ over $A$ is isomorphic to a bimodule associated with a finite-index endomorphism as in Example \ref{exa:bimodule}.
\end{lem}
\begin{proof}

First, we assume that $A$ is stable.
    Set a right Hilbert \Cs-module $\ell^2(A)_{A}:=A\otimes\ell^{2}$.
    Since $A$ is separable, $X_{A}$ is a direct summand of $\ell^2(A)_{A}$ \cite[Theorem 6.2]{MR1325694}.
    As in Remark \ref{rmk:inclusion_vs_bimodule} (see also \cite[Theorem 2.22]{MR2085108}),
    $_{A}X_{A}$ is a proper right \Cs-correspondence over a $\sigma$-unital \Cs-algebra.
    Hence,
    there is a projection $p\in\End(\ell^2(A)_{A})$,
    and a nondegenerate $*$-homomorphism $\phi\colon A\to \CPT(p\ell^2(A)_{A})$ such that $_{A}X_{A}$ is isomorphic to the \Cs-correspondence associated with $((pH)_{A},\;\phi)$.
    Since $\CPT(p\ell^2(A)_{A})=p\CPT(\ell^2(A)_{A})p$ contains a stable \Cs-subalgebra $\phi(A)$ nondegenerately, it is stable.
    By the construction of $\ell^2(A)_{A}$,
    we can take a projection $q\in\End(\ell^2(A)_{A})$ such that $q\ell^2(A)_{A}$ is isomorphic to the trivial right module $A_{A}$.
    We see that $\CPT(\ell^2(A)_{A})\cong A\otimes\IK$ is simple, in particular, $p\CPT(\ell^2(A)_{A})p$ and $q\CPT(\ell^2(A)_{A})q$ are stable full corners of $\CPT(\ell^2(A)_{A})$.
    Thanks to \cite[Theorem 4.23]{MR969204},
    there is a partial isometry $V\in \End(\ell^2(A)_{A})=M(\CPT(\ell^2(A)_{A}))$ satisfying $V^{*}V=p$ and $VV^{*}=q$.
    We get the $*$-homomorphism 
    \[\rho\colon A\to q\CPT(\ell^2(A)_{A})q=\CPT(q\ell^2(A)_{A})\cong A\]
    defined as $\rho(x):=V\phi(x)V^{*}$.
   By construction,
   $_{A}X_{A}$ is unitary isomorphic to the \Cs-correspondence ${}_{\rho}A$ associated with $\rho$ (see Example \ref{exa:bimodule}).
   Thanks to Remark \ref{rmk:inclusion_vs_bimodule},
   there is a finite-index conditional expectation $E\colon \CPT(({}_{\rho}A)_{A})=A\to \rho(A)$,
   then $\rho$ is of finite index. 

   When $A$ is unital, purely infinite simple, and in the Cuntz standard form,
   we can consider the right module $\ell^2(A)_{A}$, projections $p,q\in\End(\ell^2(A)_{A})$, and a $*$-homomorphism $\phi$ as in the previous paragraph.
   That is, ${}_{A}X_{A}$ is isomorphic to the \Cs-correspondence associated with $\phi\colon A\to \CPT(p\ell^2(A)_{A})\cong p\CPT(\ell^2(A)_{A})p$, and $A\cong\CPT(q\ell^2(A)_{A})\cong q\CPT(\ell^2(A)_{A})q$.
   Notice that there are unital $*$-homomorphisms from $A$ to $p\CPT(\ell^2(A)_{A})p$ and $q\CPT(\ell^2(A)_{A})q$ given by the left actions. 
   As these maps send $1_{A}$ to $p$ and $q$ respectively, and $0=[1_{A}]_{0}\in K_{0}(A)$,
   we get $[p]_{0}=[q]_{0}=0$ in ${\rm K}_{0}(\CPT(\ell^2(A)_{A}))$.
   Since $\CPT(\ell^2(A)_{A})\cong A\otimes\IK$ is purely infinite simple,
   $p$ and $q$ are Murray--von Neumann equivalent in $\CPT(\ell^2(A)_{A})$.
   Similar to the previous paragraph,
   taking a partial isometry $V\in \CPT(\ell^2(A)_{A})$ with $V^{*}V=p$ and $VV^{*}=q$, we get the endomorphism 
   \[\rho\colon A\to q\CPT(\ell^2(A)_{A})q\cong A\]
   defined by $\rho(x)=V\phi(x)V^{*}$.
\end{proof}

\begin{rmk}
    As a consequence of Lemma \ref{lem:stable_bimodule}, if $A$ is a separable simple \Cs-algebra which is also stable or $\cO_2$-stable then every finite-index $A$-$A$ bimodule $X$ can be realized by an endomorphism $\rho: A\to A$ of finite index as ${}_{A}({}_\rho A)_A\cong {}_AX_A$ unitarily.
    Therefore, as unitary tensor categories, we have that 
\begin{align*}
    \fiBim(A)\simeq\fiEnd(A).
\end{align*}
Here,
    we remark that ${}_{\rho_{2}}A\boxtimes_{A}{}_{\rho_{1}}{A}$ corresponds to $\rho_{1}\circ\rho_{2}$, and this also corresponds to the product of sectors as defined in Notation \ref{notation:sectors}.
    However, giving a detailed proof of this fact would take us far afield, and we shall not use it. 
\end{rmk}

We now describe the relevant functors between unitary tensor categories: 
A {\bf unitary tensor functor} is a triple $(F,\ F^1,\ F^2)$ consisting of a $*$-functor $F:\cC\to \cD,$ a chosen unitary isomorphism $F^1\in\cD(1_\cD\to F(1_\cC))$ (which we will hereafter assume is the identity), and a unitary natural isomorphism $F^2=\{F^2_{a,b}: F(a)\otimes_\cD F(b)\to F(a\otimes_\cC b)\}_{a, b\in\cC},$ sometimes called the \emph{tensorator} of $F.$ We shall often refer to $F$ as a unitary tensor functor without explicitly mentioning $F^2$, and we will always suppress the use of $F^1$.

\smallskip

An {\bf action of a UTC $\cC$ on a (not necessarily unital) \Cs-algebra $A$} is a \emph{unitary tensor functor}
 \footnote{For consistency with group actions, actions of tensor categories are often defined as functors $\cC^{\rm op}\to\fiBim(A)$. In this paper, however, we adopt the above definition. The main results are independent of this choice.}
\begin{align}
    F:\cC\to \fiBim(A)
\end{align}
We say that \textbf{an action $F$ is outer} if it is given by a fully-faithful functor. That is, if $F$ is bijective at the level of Hom spaces. We notice that unitarity and semisimplicity make $F$ automatically faithful, however fullness is far from automatic. 

\smallskip
There are a few main sources of UTC actions that we will be interested in:
\begin{exa}\label{exa:HilbGamma}
    Let $\Gamma$ be a countable discrete group, and consider \textbf{the UTC of $\Gamma$-graded Hilbert spaces} denoted $\fdHilb(\Gamma)$. That is, objects  $\cH\in\fdHilb(\Gamma)$ are Hilbert spaces admitting a grading $\cH = \oplus_{g\in \Gamma}\cH_g,$ where each $\cH_g$ is a finite-dimensional Hilbert space and the direct sum is finitely supported on $\Gamma.$ A standard choice of irreducibles is $\Irr(\fdHilb(\Gamma)) = \{\IC_g\}_{g\in \Gamma}$, and so $\cH\cong \oplus_g \IC_g^{\mathsf{dim}(\cH_g)}.$
    The morphisms are bounded linear maps preserving the grading. 
    The tensor product structure can be expressed as (the linear extension of) $\IC_g\otimes \IC_h\cong \IC_{gh}$ for $g,h\in\Gamma,$ encoding the group structure.

    There is a variant $\fdHilb(\Gamma, \omega)$, the UTC of \textbf{ twisted $\Gamma$-graded Hilbert spaces}, where $\omega$ is a normalized $3$-cocycle (determined up to a coboundary) giving the data of the associator $\alpha_{g,h,k} = \omega(g,h,k)\cdot\id$,  satisfying the pentagonal equations defining a tensor category. 
    All the other structures remain the same. 

    Given a separable \Cs-algebra, an action $\alpha: \Gamma\acts A$ defines a UTC action 
    \begin{align*}
        F:\fdHilb(\Gamma)^{\op}&\to\fiBim(A)\qquad \text{with tensorator} \qquad F^2_{g,h}:F(\IC_g)\boxtimes F(\IC_h)\to F(\IC_{hg})\\
        \IC_g&\mapsto {}_gA_A \qquad\hspace{6cm} \eta\boxtimes\xi\mapsto\ \alpha_{h}(\eta)\xi,
    \end{align*}
    where ${}_gA_A$ is the right Hilbert $A$-module $A_A$ with left action $a\rhd\zeta = \alpha_g(a)\zeta.$ Outerness of $\alpha$ is equivalent to full-faithfulness of $F$. 
    We recall that $\fdHilb(\Gamma)^{\op}$ denotes the same category of $\Gamma$-graded Hilbert spaces as above, but with the opposite tensor product. 
\end{exa}

Very many examples of interest arise from (the discrete duals of) compact quantum groups, which we summarize below.
\begin{exa}\label{exa: RepCompact}
    Let $G$ be a compact group, and we consider \textbf{the UTC of finite-dimensional representations/modules over $G$}, denoted $\Rep(G)$, whose morphisms are intertwiners. The tensor product on objects is the usual tensor product of representations, and on morphisms is the usual tensor product of linear maps. 
    
    Given an action $\alpha$ of $G$ on a separable \Cs-algebra $A$, we consider \emph{the crossed product \Cs-algebra} $A\rtimes_{\alpha}G$  as well as the \emph{fixed-point algebra} $A^G$.
    Given such a representation $(\pi, \cH_\pi)$ we can consider the right Hilbert $A\rtimes_{\alpha}G$-module $\cH_{\pi}\otimes A\rtimes_{\alpha}G$ and left action $A\rtimes_{\alpha}G\to\End(\cH_{\pi}\otimes A\rtimes_{\alpha}G_{A\rtimes G})\cong B(\cH_{\pi})\otimes M(A\rtimes_{\alpha}G)$ associated with the canonical inclusion of $A$ and the unitary representation $\pi\otimes\lambda$ of $G$, where $\lambda$ is the left regular representation $G\to M(A\rtimes G)$.
    The associated bimodule is denoted by $X_{\pi}\in\fiBim(A\rtimes G)$.
    This defines the action $F_{\alpha}\colon\Rep(G)\ni(\pi,\cH_{\pi})\mapsto X_{\pi}\in\fiBim(A\rtimes G)$ of $\Rep(G)$ on $A\rtimes G$.
    If $\alpha$ is \emph{minimal} (i.e., $M(A)\cap (A^{G})'\cong \IC$) and the fixed-point algebra is simple,
    then $A\rtimes G$ is a simple \Cs-algebra \cite[Proposition A]{MR4813137}.
    In this case, $A^{G}$ is isomorphic to the full corner of $A\rtimes G$ given by the averaging projection $p_{G}:=\int_{G}\lambda_{g}dg$, and this implies $A^{G}$ is Morita equivalent to $A\rtimes G$. 
    Hence $F_{\alpha}$ defines \textbf{the dual action of $\Rep(G)$ on $A^{G}$}.
    
    We remark that this example closely generalizes to compact quantum groups, 
    however we shall not expand the details here. 
\end{exa}

There are many other outer actions of UTCs on \Cs-algebras that do not necessarily come from compact quantum groups or their discrete duals. 
To name a few, given an arbitrary UTC $\cC$, Hartglass and the first-named author constructed a separable monotracial exact simple unital \Cs-algebra $A=A(\cC)$ and an outer action $\cC\acts A$ in \cite{MR4139893}. Moreover, the first-named author showed in \cite{2025arXiv250321515H} how to obtain UTC-actions using the gauge symmetries of certain Cuntz-Pimsner algebras. 
By different methods, Izumi constructed families of UTC-actions by sectors of Cuntz-Krieger \Cs-algebras in \cite{MR1228532, MR1604162}. More recently, Kitamura constructed outer actions of arbitrary UTCs on Kirchberg algebras in \cite{2024arXiv240518429K} using their structural properties.

\subsection{\texorpdfstring{\Cs-discrete inclusions}{}}
In this section, we summarize the theory of \Cs-discrete inclusions that we shall need, developed in \cite{doi:10.1142/S1793525325500256, 2024arXiv240918161H, 2025arXiv250321515H} and references therein. 
Discreteness for subfactors was originally introduced by Izumi, Longo, and Popa in \cite[Definition 3.7]{MR1622812}. 
Later, C. Jones and Penneys introduced a tensor categorical description of discrete subfactors in \cite{MR3948170}, focusing on actions of UTCs on $\rm{II}_1$-factors and highlighting the importance of so-called \emph{\Cs-algebra objects}.

Continuing with the notation from Example \ref{exa:basic_construction}, we consider a nondegenerate inclusion $B\subset A$, equipped with a faithful conditional expectation $E\colon A\to B$. 
That is, $E$ is a $B$-$B$ bimodular completely positive faithful map onto $B$. If $B$ is unital, then by nondegeneracy, $A$ is also unital with the same unit and thus $E$ is unital as well. 
When referring to the inclusion, we shall often omit $E$ from the notation when it is not explicitly needed, but all inclusions we consider come equipped with such $E$ unless specified otherwise. 
We say that $B\subset A$ \textbf{is irreducible} in case $B'\cap M(A)\cong \IC.$

The \textbf{projective quasi-normalizer} (PQN) of $A\subset B$  is defined as 
\begin{align}\label{eqn:PQN}
    A^\diamondsuit:=\PQN(B\subset A):=\{a\in A|\ \exists K\in \fiBim(B),\ \eta(a)\in K\subset \eta(A) \},
\end{align}
which by \Cs-Frobenius Reciprocity \cite[Theorem 3.6]{doi:10.1142/S1793525325500256} in the unital case is automatically an intermediate $*$-algebra $B\subset A^\diamondsuit\subset A.$ 
We sometimes refer to elements in $A^\diamondsuit$ as \textbf{projective quasi-normalizers}.

\begin{defn}\label{defn:C*Disc}
The inclusion $B\subset A$ is called \textbf{\Cs-discrete} if $A^\diamondsuit\subset A$ is dense in norm. 
\end{defn}

We shall now introduce the tensor categorical aspects of \Cs-discrete inclusions. 
A {\bf $*$-algebra object in $\cC$} is a \emph{lax} tensor functor $(\IA, \IA^2):\cC^{\op}\to \Vec$ into the category of (not necessarily finite-dimensional) complex vector spaces equipped with a $*$-structure given by a conjugate-linear natural transformation $\{j^{\IA}:\IA(c)\to \IA(\overline{c})\}_{c\in\cC}$ which is involutive, unital, and reverses multiplication.\footnote{We are really talking about algebra objects in $\Vec(\cC),$ the category of linear functors $\cC^{\op}\to \Vec$ with natural linear transformations, but we shall obfuscate this detail from our discussion.}
(For further details on the $*$-structure, we refer the reader to  \cite[Definition 22]{JP17}). 
That $(\IA,\IA^2)$ is lax tensor means that the tensorator $\IA^2$ natural transformation  
$\{\IA^2_{a,b}: \IA(a)\otimes \IA(b)\to \IA(a\otimes b)\}_{a,b\in \cC}$, consisting of not necessarily isomorphisms is unital and associative. 
For further details, we direct the reader to \cite[\S2.2]{JP17}. 
Furthermore, we say a $*$-algebra object is {\bf connected} if $\IA(1_\cC)\cong \IC.$

As proven in \cite[Theorem 2]{JP17}, $*$-algebra objects in $\cC$ correspond to \emph{cyclic} $\cC$-module $*$-categories. 
That is, a $*$-category $\cM$ together with a chosen object $m\in \cM$ and a $\cC$-module structure on $\cM$ such that every object in $\cM$ has the form $c\otimes m$ for $c\in \cC$, yield the same data as a $*$-algebra object $\IA.$
A $*$-algebra object $\IA$ is a {\bf \Cs-algebra object} if its corresponding $\cC$-module $*$-category $\cM$ is a \Cs-category (i.e. every $*$-algebra of endomorphisms in $\cM$ is in fact a \Cs-algebra) \cite[Definition 25]{JP17}.  

\begin{exa}
    Consider the UTC $\fdHilb(\Gamma, \omega),$ with the notation from Example \ref{exa:HilbGamma}. 
    By \cite[Proposition 6.1]{MR3948170}, connected \Cs-algebra objects $\IA$ in $\fdHilb(\Gamma, \omega)$ are classified up to $*$-algebra isomorphism by pairs $(\Lambda, [\mu])$, for a subgroup $\Lambda\leq\Gamma$ such that $\omega|_\Lambda$ is cohomologically trivial, and a normalized $2$-cocycle $[\mu]\in H^2(\Lambda, U(1))$ satisfying $\mu(e, g) =  1 = \mu(g, e)$ for all $g\in \Lambda.$ 
\end{exa}
\begin{exa}
    We now describe examples of connected \Cs-algebra objects from a given compact quantum group $\IG$ \cite[Examples 16]{JP17}. 
    For details on compact quantum groups and their representations, we refer the reader to \cite{MR3204665}. 
    By Tannaka-Krein duality, $\IG$ can be recovered from the data of its finite-dimensional representation category $\Rep(\IG)$ (c.f. Example \ref{exa: RepCompact}) together with the forgetful \emph{fiber functor} $\mathsf{For}:\Rep(\IG)\to \fdHilb,$ which forgets the actions of $\IG$ and gives the underlying Hilbert spaces. 
    The forgetful functor endows $\Hilb$ with the structure of a $\Rep(\IG)$-module \Cs-category, and considering the base-point $m=\IC\in \Hilb,$ by \cite[Theorem 2]{JP17} we obtain a connected \Cs-algebra object $\IA$ in $\Rep(\IG).$
\end{exa}
\begin{exa}
Let $B$ be a \Cs-algebra with a UTC action $F:\cC\to \fiBim(B).$ 
Then, $\cM:=\fiBim(B)$ is a $\cC$-module \Cs-category structure with action determined by $F$.
By choosing the base-point $m= {}_BB_B$, we obtain the connected \Cs-algebra object $\mathsf{For}\circ F:\cC\to \Vec$ determined by \cite[Theorem 2]{JP17}. 
Here, $\mathsf{For}$ is the functor forgetting the correspondence structure and keeping the vector space structure. 
To ease the notation, we shall continue denoting the \Cs-algebra object $\mathsf{For}\circ F$ simply by $F$.
For further details regarding the $*$-algebra object structure on $F$, we direct the reader to \cite[\S3.2, 3.3 and the proof of Theorem 2]{JP17}.
\end{exa}

We now sketch a reduced crossed product construction for a UTC action on a unital \Cs-algebra with trivial center over a connected \Cs-algebra object, which gives a well-understood class of unital irreducible \Cs-discrete inclusions \cite{{doi:10.1142/S1793525325500256}}. 
Given a unital \Cs-algebra $B$ with $B\cap B'\cong\IC$ and a UTC action $F:\cC\to  \fiBim(B)$ and a connected \Cs-algebra object $\IA$, we shall construct the  reduced crossed-product \Cs-algebra of $B$ by $F$ under $\IA$ (c.f. \cite[Construction 4.4]{doi:10.1142/S1793525325500256}). 
First, consider {\bf algebraic realization/crossed-product of $B$ by $F$ under $\IA$} whose underlying vector space is
\begin{align}\label{eqn:AlgCrossedProduct}
    B\rtimes_{\text{alg}, F}\IA:=\bigoplus _{c\in\Irr(\cC)}F(c)\otimes_{\IC} \IA(c), 
\end{align}
which has a canonical $*$-algebra structure coming from those of $F$ and $\IA$.
Indeed, given finite sums $x=\sum_c \xi_{(c)}\otimes f_{(c)}$ and $y = \sum_c \eta_{(c)}\otimes g_{(c)}$ in $B\rtimes_{\text{alg}, F}\IA$, the multiplication and involution are given by 
\begin{align}
    xy &= \sum_{c,d,e\in\Irr(\cC)}\sum_{\alpha\in\Isom(e,c\otimes d)}F(\alpha^*)\circ[F^2_{c,d}(\xi_{(c)}\otimes \eta_{(d)})]\otimes \IA(\alpha)\circ[\IA^2_{c,d}(f_{(c)}\otimes g_{(d)})], \nonumber \\
     x^* &= \sum_{c\in \Irr(\cC)} j^F_c(\xi_{(c)})\otimes j^{\IA}_c(f_{(c)}).
\end{align}
Here, $\Isom(e,c\otimes d)$ denotes a chosen complete set of isometries $\alpha: e\to c\otimes d$ with orthogonal ranges. 
That is, $\sum_{\alpha\in\Isom(e,c\otimes d)}\alpha\circ\alpha^* = \id_{c\otimes d}.$

There is a faithful unital completely positive surjective map $E':B\rtimes_{\text{alg}, F}\IA\to B$ given by the projection onto the $1_\cC$-graded component $B$. 
We consider the $B$-valued inner product 
$$
\langle x|\ y \rangle_B:= E'(x^*y),
$$
and let $\cE'_B$ be the right Hilbert $B$-module obtained as the separation completion of $A$ with respect to this inner-product. 
Faithfulness of $E'$ ensures there is an injective bounded map 
$$
\eta': A\to \cE'
$$
with dense range. 

There is a left $B\rtimes_{\text{alg}, F}\IA$-action on $\eta'(B\rtimes_{\text{alg}, F}\IA)$ by bounded right $B$-linear operators given by $x\rhd \eta'(y) := \eta'(xy)$. 
Therefore, the action of $x$ extends to a bounded operator on $\cE'$ which is readily seen to be adjointable.
That is, there is an embedding 
$
B\rtimes_{\text{alg}, F}\IA\hookrightarrow \End(\cE'_B).
$
We then define the {\bf reduced crossed-product \Cs-algebra of $B$ by $F$ under $\IA$} by 
\begin{align}
    B\rtimes_{r, F}\IA :=\overline{B\rtimes_{\text{alg}, F}\IA}\subset \End(\cE'_B). 
\end{align}
\noindent Finally, $E'$ extends to a unital completely positive faithful $B$-$B$ bimodular surjective map $E':B\rtimes_{r,F}\IA\to B$, yielding an irreducible \Cs-discrete extension.

In \cite[Theorem 1.1, Corollary 1.2]{doi:10.1142/S1793525325500256}, the first-named author and Nelson characterized unital irreducible \Cs-discrete inclusions in terms of an outer UTC action $F:\cC\acts B$ with a chosen connected \Cs-algebra object $\IA$, given by {\bf the diamond spaces}.
That is, for $c\in \cC,$ we have 
\begin{align}\label{eq:diamond}
\IA(c) \cong \rCorr({}_BF(c)_B\to {}_B\cE_B)^\diamondsuit:=\left\{f\in\rCorr({}_BF(c)_B\to {}_B\cE_B)|\ f(F(c))\subset \eta(A)\right\}.    
\end{align}
Such inclusions can therefore be expressed as a reduced crossed product
$$
(B\subset A) \cong (B\subset B\rtimes_{r,F}\IA).
$$
This class, therefore, includes all crossed products by outer actions of discrete (quantum) groups on unital \Cs-algebras, all irreducible finite-index inclusions, as well as certain semicircular systems \cite{2024arXiv240918161H}, and cores of Cuntz-Pimsner algebras \cite{2025arXiv250321515H}.

\subsection{Properly outer endomorphisms}\label{sec:quasi_product}
In this article, we consider properly outer endomorphisms of simple \Cs-algebras.
We use the following definition (see \cite{MR1234394},
\cite[Section 7]{MR1900138}, 
\cite[Section 2.2]{doi:10.1142/S0129055X24610026} for details).
\begin{defn}[See Section 1 of \cite{MR1234394}]\label{defn:prop_outer}
    An endomorphism $\rho$ of a \Cs-algebra $A$ is {\bf properly outer} if for every nonzero hereditary subalgebra $H$ of $A$,
    every $a\in A$,
    and every $\epsilon>0$,
    there exists a positive element $c$ in $H$ with $\|c\|=1$ satisfying $\|ca\rho(c)\|<\epsilon$.
\end{defn}
Proper outerness can be considered as a noncommutative analogue of freeness as follows. 
\begin{exa}\label{exa:top_free}
    Let $A=C_{0}(X)$ be a commutative \Cs-algebra and $\alpha\in{\rm Aut}(C_{0
    }(X))$.
    We say that $\alpha\colon \IZ\acts X$ is {\bf topologically free} if for any open subspace $\emptyset\neq U\subset X$ and $s\in\IZ\setminus\{0\}$,
    there exists $x\in U$ such that $\alpha_{s}(x)\neq x$,
    equivalently, there exists a positive norm-one function $f\in C_{0}(U)$ with $f\alpha_{s}(f)=0$. 
    Since every hereditary subalgebra of $C_{0}(X)$ has a form $C_{0}(U)$ for some open subspace $U$,
    we can easily check that $\{\alpha^{n}\}_{n\in\IZ\setminus\{0\}}$ are properly outer if and only if the associated action $\IZ\acts C_{0}(X)$ is topologically free.
\end{exa}
We mention the following lemma.
The proof of \cite[Lemma 3.2]{MR634163} works for properly outer endomorphisms.
According to the proof of \cite[Lemma 3.2]{MR634163},
the following statement holds for a self-adjoint element $a$ not necessarily positive.
\begin{lem}[Lemma 3.2 of \cite{MR634163}]\label{lem:prop_outer}
Let $a$ be a self-adjoint element of a \Cs-algebra $A$,
$x_{1},\dots,x_{n}$ be elements of $A$,
and $\rho_{1},\dots,\rho_{n}$ be properly outer endomorphisms of $A$.
Then,
for every $\epsilon>0$,
there is a positive element $c\in A$ with $\|c\|=1$ such that 
\[
\|cac\|>\|a\|-\epsilon,\;{\rm and}\;\|cx_{i}\rho_{i}(c)\|<\epsilon
\]
hold for every $i=1,\dots,n$.
\end{lem}

Proper outerness for endomorphisms of simple \Cs-algebras has been studied in different variations.
\begin{thm}[Lemma 1.1 of \cite{MR634163}]
    Every outer automorphism of a simple \Cs-algebra is properly outer.
\end{thm} 

 In many cases of interest, the previous theorem has also been established, even by dropping the surjectivity of an endomorphism of $A$. 
 For a finite-index endomorphism $\rho\colon A\to A,$ being surjective is equivalent to having statistical dimension $d(\rho)=1$. 
 Indeed, given an expectation $E\colon A\to \rho(A)$ with $E(a)\geq a$ for all $a\in A_{+}$ we have $\|E(a)-a\|\leq \|E(E(a)-a)\| = 0.$ 
 This implies $E= \id_A, $ and $\rho$ is surjective.

For finite-index endomorphisms of finite depth which are not surjective, Izumi proved that proper outerness is automatic.
\begin{thm}[Theorem 7.5 of \cite{MR1900138}]
    Let $A$ be a simple $\sigma$-unital \Cs-algebra and $\rho$ be a finite-index irreducible nonsurjective endomorphism.
    If $\rho$ is of finite depth, then $\rho$ is properly outer. 
\end{thm}

Furthermore, the assumption of finite depth can be dropped, provided $A$ is purely infinite simple.

\begin{thm}[Corollary 3.4 of \cite{doi:10.1142/S0129055X24610026}]\label{thm:properly_outer}
    Let $A$ be a separable purely infinite simple \Cs-algebra and $\rho$ be a finite-index irreducible nonsurjective endomorphism,
    then $\rho$ is properly outer.
\end{thm}

\section{Proof of Theorem \ref{thmalpha:end}}
The next theorem is our main result, which shows that the conclusion of \cite[Corollary 3.4]{doi:10.1142/S0129055X24610026} holds without assuming pure infiniteness.
\begin{thm}[{Theorem \ref{thmalpha:end}}]\label{thm:end}
 Let $A$ be a separable simple \Cs-algebra and $\rho\colon A\to A$ be a finite-index irreducible nonsurjective endomorphism.
 Then $\rho$ is properly outer.
\end{thm}

Let $A$ and $\rho$ be as above. 
Since $\rho$ is properly outer if and only if $\rho\otimes\id_{\IK}$ is properly outer, we may and do assume $A$ is stable.
Hence, $\rho$ generates a unitary tensor category in the category $\Sect(A)$ of sectors.
(See Section 4 of \cite{MR1900138}.)

To prove Theorem \ref{thmalpha:end}, we will adapt similar results and their proofs from \cite{MR1900138} to work in the \emph{infinite depth} scenario (i.e. we do not assume that $\rho$ generates a fusion category).
We shall not reproduce these results here, and moreover, we will therefore try to adjust our notation to resemble Izumi's notation from that article as closely as possible. 

We assume $\rho$ is not properly outer for the sake of contradiction.
As in the proof of \cite[Theorem 7.5]{MR1900138},
there is an irreducible representation $\pi\colon A\to B(\cH)$,
a norm-one element $a\in A$,
a nonzero hereditary subalgebra $H\subset A$,
an isometry $V\in(\pi\circ\rho,\pi)\subset B(\cH)$,
a compact operator $T\in\IK(\cH)$,
and $\delta>0$ satisfying
\begin{equation}\label{eq:contradict}
\pi(h)\left(\pi(a)V+V^{*}\pi(a^{*})+T\right)\pi(h^*)\geq\delta\pi(hh^{*})
\end{equation}
for all $h\in H$. 
The alert reader will notice that the statement of \cite[Theorem 7.5]{MR1900138} assumes finite depth. 
However, to obtain Inequality (\ref{eq:contradict}), that assumption was not used, and so it holds in the general UTC setting. 
We notice that this inequality was already applied in the possibly infinite-depth case of $\Rep(G)$ for a compact group $G$ by Izumi in \cite[Section 4]{doi:10.1142/S0129055X24610026}.

\begin{rmk}
    From the above statement,
    we see that the representation $\pi\circ\rho$ contains $\pi$ as a direct summand.
    Roughly speaking,
    as in Example \ref{exa:top_free},
    non-proper outerness implies non-freeness of the action on the space of irreducible representations.
    In fact,
    it is known that if $A$ is separable simple \Cs-algebra and $\rho\in \End(A)$ is not properly outer,
    then every irreducible representation $\sigma$ of $A$ is contained in $\sigma\circ\rho$ (see \cite[Lemma 2.2]{doi:10.1142/S0129055X24610026}). 
\end{rmk}
We introduce the following notation:
\begin{nota}\label{notation:Ddiamond}
\begin{enumerate}[label=(\roman*)]
    \item We consider the purely infinite and simple \Cs-algebra \[B:=\cO_{\infty}\otimes \pi(A)\subset \cO_{\infty}\otimes B(\cH).\footnote{Tensoring by any Kirchberg algebra would do. However, we chose $\cO_\infty$ to keep the notation as simple as possible.}\]
    \item We consider the \Cs-subalgebra of $\cO_\infty\otimes B(\cH)$ generated by $B(1_{\cO_\infty}\otimes V)$  
    \begin{align*}
    D:=C^*\{B\cdot(1_{\cO_\infty}\otimes V)\}\subset \cO_\infty\otimes B(\cH),
    \end{align*}
    which contains the dense $*$-algebra $D_{\rm alg}$ given by 
    \begin{align*}
    {\rm span}_{\IC}\left\{b_{1}(1_{\cO_\infty}\otimes V)^{k_{1}}(1_{\cO_\infty}\otimes V^*)^{l_{1}}b_{2}\cdots b_{n}(1_{\cO_\infty}\otimes V)^{k_{n}}(1_{\cO_\infty}\otimes V^*)^{l_{n}}b_{n+1}\Bigg|\ b_{i}\in B\right\}.
    \end{align*}

    \item Let $\{\rho_{t}\}_{t\in\cT}$ be a complete system of representatives of irreducible sectors of $B$.
    We assume that for each $t\in\cT$, there is $\overline{t}\in\cT$ such that $\overline{\rho}_{t}=\rho_{\overline{t}}$.
    We put $\rho_{e}:=\id_{B}$ and $\rho_{t_{0}}:=\id_{\cO_{\infty}}\otimes\rho$.

    \item For each $t\in\cT$, set 
    $$\cL_{t}:=(\iota\circ\rho_{t},\iota)\subset M(D),$$ 
    where 
    $$\iota\colon B\to D$$ 
    is the trivial inclusion: 
    \begin{align*}
    B= (1_{\cO_\infty}\otimes V)^*(1_{\cO_\infty}\otimes V)B &= (1_{\cO_\infty}\otimes V)^*\rho_{t_0}(B)(1_{\cO_\infty}\otimes V)\\&\subset (1_{\cO_\infty}\otimes V)^*B(1_{\cO_\infty}\otimes V)\subset D.
    \end{align*}
    By Lemma 5.2 of \cite{MR1900138}, $\cL_{t}$ is finite dimensional, and we set $\dim \cL_{t}=n_{t}$. 
    Notice that $(1_{\cO_\infty}\otimes V)\in \cL_{t_0}.$ 
    
    \item We put $d(t):=d(\rho_{t})$ and $R_{t}:=R_{\rho_{t}}$ (see Lemma \ref{conjugate}) for each $t\in\cT$.
    Since $\overline{R}_{t}^{*}\rho_{t}(R_{t})=R_{t}^{*}\rho_{\overline{t}}(\overline{R}_{t})=d(t)^{-1}$, the following antilinear maps are mutually inverse:
    \[\cL_{t}\ni S\mapsto \sqrt{d(t)}S^{*}\overline{R}_{t}\in\cL_{\overline{t}},\quad\cL_{\overline{t}}\ni T\mapsto \sqrt{d(t)}T^{*}R_{t}\in\cL_{t}.\] We may and do assume that $R_{\overline{t}}=\overline{R}_{t}$.\label{item:Rd}
    \item Taking any irreducible representation $\cO_{\infty}\subset B(\cK)$, 
    we see that $\cO_{\infty}\otimes\pi(A)\subset B(\cK\otimes \cH)$ is irreducible. 
    Since we have $B\subset D\subset M(D)\subset B(\cK\otimes \cH)$,
    we get $M(D)\cap B^{\prime}=\IC$.
    Now, for each $t\in\cT$, we consider the following $\IC$-valued inner-products on $\cL_t:$ 
    \begin{align*}
        \langle W|S\rangle_{t}1_{M(D)}&:=W^{*}S,\\
{}_{t}\langle S,W\rangle1_{M(D)}&:=d(t)R_{\overline{t}}^{*}SW^{*}R_{\overline{t}}\qquad \forall S,W\in\cL_{t}.
    \end{align*}
    \item For each $t\in\cT$,
    let $\{V(t)_{i}\}_{i=1}^{n_{t}}$ be an orthonormal basis of $\cL_{t}$ (resp. $\{\tilde{V}(t)_{i}\}_{i}$) with respect to the right (resp. left) inner product $\langle\cdot|\cdot\rangle_{t}$ (resp. ${}_{t}\langle\cdot,\cdot\rangle$). 
    \item We put 
    $$D^{\diamondsuit}:= {\rm span}_{\IC}\{B\cdot \cL_t\}_{t\in \cT}.$$
\end{enumerate}    
\end{nota}

\begin{lem}\label{lem_alg}
The following hold.
    \begin{enumerate}
        \item $D^{\diamondsuit}$ is a dense $*$-algebra of $D$.
        \item If we have $\sum_{t\in F}\sum_{i=1}^{n_{t}}x(t)_{i}V(t)_{i}=0$ for some finite set $F\subset\cT$ and $x(t)_{i}\in B$,
        then we get $x(e)=0$.
        \item The linear map $E_{0}\colon D^{\diamondsuit}\to B$ defined as follows is completely positive and faithful:
        \[E_{0}\bigg(\sum_{t\in F}\sum_{i=1}^{n_{t}}x(t)_{i}V(t)_{i}\bigg):=x(e)\;\text { for all }\;F\subset_{\rm fin}\cT\;\text{ and }\;x(t)_{i}\in B.\]
        \item The following holds for all $s,t\in\cT$ and $x, y\in B$:
        \begin{equation*}
             E_{0}(x\tilde{V}(t)_{i}\tilde{V}(s)_{j}^{*}y)    =\delta_{s,t}\delta_{i,j}d(s)^{-1}xy.
        \end{equation*}
    \end{enumerate}
\end{lem}
Item (1) follows similarly to the proof of \cite[Lemma 7.1]{MR1900138}.
Since our claim does not require a detailed computation, we also give a brief proof for completeness.
\begin{proof}  
    (1): We see that $BD^{\diamondsuit}B \subset D^{\diamondsuit}$. 
    Since $D^{\diamondsuit}$ contains the generators of $D_{\rm alg}$ (see Notation \ref{notation:Ddiamond} (ii)) as a $*$-algebra,
    it is enough to show that $D^{\diamondsuit}$ is closed under taking $*$ and multiplication of $\cL_{t}$.
    Let $R_{t}\in(\overline{\rho}_{t}\circ\rho_{t},\;\id_{B})$ and $\overline{R}_{t}\in(\rho_{t}\circ\overline{\rho}_{t},\;\id_{B})$ be the isometries from Notation \ref{notation:Ddiamond} (v),
    then we have
    \[
        \cL_{t}^{*}B
        =\cL_{t}^{*}(d(t)^{-1}B)
        =\cL_{t}^{*}\rho_{t}(R_{t}^{*})\overline{R}_{t}B
        =R_{t}^{*}(\cL_{t}^{*}\overline{R}_{t})B
        =R_{t}^{*}\cL_{\overline{t}}B
        =R_{t}^{*}\rho_{\overline{t}}(B)\cL_{\overline{t}}=B\cL_{\overline{t}}.
    \]
    Here, the second equality follows from $d(t)^{-1}=\overline{R}_{t}^{*}\rho_{t}(R_{t})$, and the last equality follows since $(d(t)R_{t}^{*},d(t)R_{t})$ is a quasi-basis of $\rho_{\bar{t}}(B)\subset B$ (see Lemma \ref{conjugate}).
    For the third equality,
    note that the nondegenerate endomorphism $\rho_{t}$ of $B$ extends to  
    the multiplier algebra, which is also denoted by $\rho_{t}$.
    It is straightforward to verify that the commutation relation between $\cL_{t}$ and $B$ extends to that between $\cL_{t}$ and $M(B)$. 
    Indeed, for any $T\in M(B),$ $b,b'\in B$ and $V\in \cL_{t}$, we have 
    \[
    \rho_{t}(b)(\rho_{t}(T)V-VT)b'=0.
    \]
    Since we have $B=\overline{B\rho_t(B)}$ and $B$ acts nondegenerately on $\cK\otimes\cH$ (see Notation \ref{notation:Ddiamond}(vi)), this implies $\rho_t(T)V=VT$ for all $T\in M(B)$.
    Hence,
    we get $(D^{\diamondsuit})^{*}=D^{\diamondsuit}$.
    Since for each $s,t\in\cT$,
    the composition $\rho_{s}\circ\rho_{t}$ is decomposed into a finite sum of irreducible endomorphisms in $\{\rho_{r}\}_{r\in\cT}$ by Lemma \ref{lem:semisimple_sectors},
    we have \[\cL_{s}\cL_{t}\subset {\rm span}\big(\bigcup_{r\in \cT}M(B)\cL_{r}\big).\]
    Hence,
    we get
    \[\cL_{t}D^{\diamondsuit},\;D^{\diamondsuit}\cL_{t}\subset D^{\diamondsuit}\]
    for every $t\in\cT$.
    Thus, $D^\diamondsuit$ is a $*$-algebra containing $D_{\rm alg}.$
\smallskip 

    To show (2),
    we assume $x:=\sum_{t\in F}\sum_{i=1}^{n_{t}}x(t)_{i}V(t)_{i}=0$ and $x(e)\neq0$.
    Replacing $x$ with $\frac{1}{\|x(e)\|^{2}}x(e)^{*}x$,
    we may assume $x(e)$ is a positive norm-one element of $B$.
    By Lemma \ref{lem:prop_outer} and Theorem \ref{thm:properly_outer},
    for any $\epsilon>0$,
    there exists a positive norm-one element $b\in B$ satisfying the following:
    \begin{align*}
        &bxb=\sum_{t\in F}\sum_{i=1}^{n_{t}}bx(t)_{i}\rho_{t}(b)V(t)_{i}\approx_{\epsilon}bx(e)b\; \text{ and }\\
        &\|bx(e)b\|>1-\epsilon.
    \end{align*}
    This contradicts $x=0$,
    then we get the claim.
\smallskip

    For (3),
    we show $E_0$ is completely positive.
    Thanks to (1) and (2),
    the linear map $E_{0}\colon D^{\diamondsuit}\to B$ is well-defined.
    Take an arbitrary $n\in \IN$, and $d=[d_{k,l}]_{k,l=1}^{n}\in M_{n}\otimes D^{\diamondsuit}$,
    with $d_{k,l}=\sum_{t\in F}\sum_{i=1}^{n_{t}}x(t)_{i}^{(k,l)}V(t)_{i}\in D^\diamondsuit$ for each $k, l= 1,2,\hdots n.$
    Let $d^{*}d=[f_{k,l}]_{k,l=1}^{n}$ then we get the following:
    \begin{align*}
      f_{k,l}&=\sum_{m=1}^{n}\sum_{s,t\in F}\sum_{i,j}V(s)_{i}^{*}x(s)_{i}^{(m,k)*}x(t)_{j}^{(m,l)}V(t)_{j}\\
      &=\sum_{m=1}^{n}\sum_{s,t\in F}\sum_{i,j}\frac{1}{d(s)}R_{s}^{*}V(s)_{i}^{*}R_{\overline{s}}x(s)_{i}^{(m,k)*}x(t)_{j}^{(m,l)}V(t)_{j}\\
      &=\sum_{m=1}^{n}\sum_{s,t\in F}\sum_{i,j}\frac{1}{d(s)}R_{s}^{*}\rho_{\overline{s}}\left(x(s)_{i}^{(m,k)*}x(t)_{j}^{(m,l)}\right)V(s)_{i}^{*}R_{\overline{s}}V(t)_{j}.
    \end{align*}
    Fixing $s,t\in F$,
    we consider the irreducible decomposition $\rho_{\overline{s}}\circ\rho_{t}(\cdot)=\sum_{r\in\cT}\sum_{p=1}^{N_{\overline{s},t}^{r}}S_{r,p}\rho_{r}(\cdot)S_{r,p}^{*}$,
    where $N_{\overline{s},t}^{r}=\dim(\rho_{\overline{s}}\circ\rho_{t},\;\rho_{r})$ and $\{S_{r,p}\mid r\in \cT\;\text{ with }\;N_{\overline{s},t}^{r}\neq0, 1\leq p\leq N_{\overline{s},t}^{r} \}$ is a finite family of isometries in $M(B)$ with the Cuntz relation $\sum_{r,p}S_{r,p}S_{r,p}^{*}=1$.
    We see that $S_{r,p}^{*}V(s)_{i}^{*}R_{\overline{s}}V(t)_{j}$ are contained in $\cL_{r}$ for all $i,j,p$,
    and $V(s)_{i}^{*}R_{\overline{s}}V(t)_{j}=\sum_{r,p}S_{r,p}\left(S_{r,p}^{*}V(s)_{i}^{*}R_{\overline{s}}V(t)_{j}\right)$.
    When $s\neq t$,
    $N_{\overline{s},t}^{e}=0$ holds.
    When $s=t$,
    we have $N_{\overline{s},s}^{e}=1$ and $S_{e,1}=R_{s}$.
    Hence,
    we have
    \begin{align*}
        E_{0}(f_{k,l})&=\sum_{m=1}^{n}\sum_{s\in F}\sum_{i,j}\frac{1}{d(s)}R_{s}^{*}\rho_{\overline{s}}\left(x(s)_{i}^{(m,k)*}x(s)_{j}^{(m,l)}\right)R_{s}\left(R_{s}^{*}V(s)_{i}^{*}R_{\overline{s}}V(s)_{j}\right)\\
        &=\sum_{m=1}^{n}\sum_{s\in F}\sum_{i}R_{s}^{*}\rho_{\overline{s}}\left(x(s)_{i}^{(m,k)*}x(s)_{i}^{(m,l)}\right)R_{s}\\
        &=\sum_{s\in F}\sum_{i}R_{s}^{*}\rho_{\overline{s}}\bigg(\sum_{m=1}^{n}x(s)_{i}^{(m,k)*}x(s)_{i}^{(m,l)}\bigg)R_{s}.
    \end{align*}
    Therefore,
    we get $(\id_{M_n}\otimes E_{0})(d^{*}d)\geq0$ and $E_{0}$ is completely positive.

    To complete (3),
    we show $E_{0}$ is faithful.
    For $d=\sum_{t\in F}\sum_{i=1}^{n_{t}}x(t)_{i}V(t)_{i}\in D^{\diamondsuit}$,
    the computations in the previous paragraph show 
    \[E_{0}(d^*d)=\sum_{s\in F}\sum_{i}R_{s}^{*}\rho_{\overline{s}}\left(x(s)_{i}^{*}x(s)_{i}\right)R_{s}.\]
   As stated in Lemma \ref{conjugate},
    we see that 
    \[E_{s}\colon B\ni x\mapsto\rho_{s}\left(R_{s}^{*}\rho_{\overline{s}}\left(x\right)R_{s}\right)\in\rho_{s}\left(B\right)\]
    is a faithful conditional expectation.
    If $d\neq0$,
    then there are $s_{0}\in F$ and $1\leq i\leq n_{s_{0}}$ with $x(s_{0})_{i}\neq0$.
    Hence,
    we get $E_{0}(d^*d)\geq R_{s_{0}}^{*}\rho_{\overline{s_{0}}}\left(x(s_{0})_{i}^{*}x(s_{0})_{i}\right)R_{s_{0}}\gneq0$.
    This shows the claim.
\smallskip
    We prove (4). 
    When $s\neq t$, in the same way as in the proof of (1),
    we have 
    \[ x\tilde{V}(t)_{i}\tilde{V}(s)_{j}^{*}y\in B\cL_{t}\cL_{s}^{*}B\subset B\cL_{t}\cL_{\overline{s}}\subset{\rm span}\bigg(\bigcup_{r\in\cT,\,(\rho_{r},\,\rho_{t}\circ\rho_{\bar{s}})\neq0}B\cL_{r}\bigg).\]
    Since $(\id,\,\rho_{t}\circ\rho_{\bar{s}})=0$,
    it follows $E_{0}(x\tilde{V}(t)_{i}\tilde{V}(s)_{j}^{*}y)=0$.
    Suppose $s=t$. 
    As in the proof of (3), consider the irreducible decomposition $\rho_{s}\circ\rho_{\bar{s}}(\cdot)=\sum_{k=1}^{m}\sum_{p=1}^{N_{s,\bar{s}}^{r_{k}}}S_{r_{k},p}\rho_{r_{k}}(\cdot)S_{r_{k},p}^{*}$ of $\rho_{s}\circ\rho_{\bar{s}}$.
    Note that $\{S_{r_{k},\,p}\}_{k\,,p}$ is a finite family of isometries in $M(B)$ with the Cuntz relation, and we have $N_{s,\bar{s}}^{e}=1$ and $S_{e,1}=\overline{R}_{s}$.
    Hence we get
    \begin{align*}
        x\tilde{V}(s)_{i}\tilde{V}(s)_{j}^{*}y
        =&d(s)x\tilde{V}(s)_{i}\tilde{V}(s)_{j}^{*}\rho_{s}(R_{s}^{*})\overline{R}_{s}y\\
        =&d(s)x\rho_{s}(R_{s}^{*})\tilde{V}(s)_{i}\tilde{V}(s)_{j}^{*}\overline{R}_{s}y\\
        =&\sum_{k,p}d(s)x\rho_{s}(R_{s}^{*})S_{r_{k},p}\left(S_{r_{k},p}^{*}\tilde{V}(s)_{i}\tilde{V}(s)_{j}^{*}\overline{R}_{s}\right)y.
    \end{align*}
    Since we have $S_{r_{k},p}^{*}\tilde{V}(s)_{i}\tilde{V}(s)_{j}^{*}\overline{R}_{s}\in\cL_{r_{k}}$ and $S_{e,1}=\overline{R}_{s} = R_{\overline{s}}$,
   the following holds: 
    \[E_{0}(x\tilde{V}(s)_{i}\tilde{V}(s)_{j}^{*}y)=d(s)x\rho_{s}(R_{s}^{*})\overline{R}_{s}\left(\overline{R}_{s}^{*}\tilde{V}(s)_{i}\tilde{V}(s)_{j}^{*}\overline{R}_{s}\right)y=d(s)^{-1}{}_{s}\big\langle\tilde{V}(s)_{i},\tilde{V}(s)_{j}\big\rangle xy.\]
    Consequently, (4) follows.
\end{proof}

\begin{rmk}\label{rmk:Ddiamond}
     Thanks to Lemma \ref{lem_alg}, we get 
     $$
     D^{\diamondsuit}\cong\bigoplus_{t\in\cT} B\otimes \cL_t
     $$
     as a $B$-$B$ bimodule, similar to the setting of Equation (\ref{eqn:AlgCrossedProduct}). 
     To see this, consider the orthonormal basis $\{$\raisebox{-0.2ex}{$\tilde{V}(t)_{i}$}$\}_{i}$ of each $\cL_{t}$ with respect to the left inner product ${}_{t}\langle\cdot,\cdot\rangle$.
     If a finite sum $d=\sum_{t}\sum_{i}x(t)_{i}$\raisebox{-0.2ex}{$\tilde{V}(t)_{i}$}$\in D^{\diamondsuit}$ is equal to zero,
     then by Lemma \ref{lem_alg} (4), we get \raisebox{-0.2ex}{$0=d(s)E_{0}(d\tilde{V}(s)_{j}^{*})=x(s)_{j}$} for every $s$ and $j$.
     This implies that the canonical surjective $B$-$B$ bilinear map $\bigoplus_{t\in\cT} B\otimes \cL_t\to D^{\diamondsuit}$ is injective.
\end{rmk}

Let $\cE$ be the \Cs-correspondence over $B$ associated with $E_{0}$ and $e_{0}\in \End(\cE_{B})$ be the Jones projection (see Example \ref{exa:basic_construction}).
Although  $D^{\diamondsuit}$ is not a \Cs-algebra,
$\cE$ can be defined by the completion of $D^{\diamondsuit}$ with respect to the norm $\|d\|_{\cE}:=\|E_{0}(d^{*}d)\|^{\frac{1}{2}}$.
Consider the canonical $*$-representation 
$$
\pi_{r}:D^{\diamondsuit}\to \End(\cE_{B}).
$$
We put 
$$
D_{r}:=\overline{\pi_{r}(D^{\diamondsuit})}\subset\End(\cE_{B}).
$$ 
There is a conditional expectation extending $E_0$ 
$$
E\colon D_{r}\to B,     
$$
determined by 
$$
E(x)e_{0}=e_{0}xe_{0}\qquad \forall x\in D_{r}.
$$
Here,
we remark that \[e_{0}D_{r}e_{0}=\overline{e_{0}\pi_{r}(D^{\diamondsuit})e_{0}}=\overline{\pi_{r}(B)}e_{0}=\pi_{r}(B)e_{0}\cong B.\]
Considering the orthonormal basis $\{\tilde{V}(t)_{i}\}_{i=1}^{n_{t}}$ of $\cL_{t}$ with respect to the left inner product $_{t}\langle\cdot,\cdot\rangle$,
we now show that the following operators $\{P_{t}\}_{t\in\cT}$ are pairwise orthogonal projections: \[P_{t}:=\sum_{i=1}^{n_{t}}d(t)\pi_{r}(\tilde{V}(t)_{i})^{*}e_{0}\pi_{r}(\tilde{V}(t)_{i})\in\End({}_{B}\cE_{B}).\footnote{Each element $\tilde{V}(t)_{i}$ is not contained in $D^{\diamondsuit}$.
However, taking approximate units $(b_{n})_{n}$ of $B$,
it is routine to show that $(\pi_{r}(\tilde{V}(t)_{i}b_{n}))_{n}$ converges to an isometry in the strong topology of $\End(\cE_{B})$. For simplicity, this limit is denoted by $\pi_{r}(\tilde{V}(t)_{i})$.}\]

\begin{lem}
Let $\eta\colon D^{\diamondsuit}\to\cE$ be the canonical inclusion map.
Then each $P_{t}\in\End({}_{B}\cE_{B})$ is the projection onto $\eta(\cL_{t}^{*}B)$.
Moreover, $\{P_{t}\}_{t\in\cT}$ is a family of pairwise orthogonal projections.
\end{lem}
\begin{proof}
For each $s,t\in\cT$ and $x,y\in B$, thanks to Lemma \ref{lem_alg} (4), we have
\[
\langle\eta(x)\mid e_{0}\pi_{r}(\tilde{V}(t)_{i}\tilde{V}(s)_{j}^{*})e_{0}\eta(y)\rangle
=E_{0}(x^{*}\tilde{V}(t)_{i}\tilde{V}(s)_{j}^{*}y)
=d(s)^{-1}\delta_{s,t}\delta_{i,j}x^{*}y.
\]
This implies
\[
e_{0}\pi_{r}(\tilde{V}(t)_{i}\tilde{V}(s)_{j}^{*})e_{0}
=d(s)^{-1}\delta_{s,t}\delta_{i,j}e_{0}.
\]
Hence we obtain
\begin{align*}
P_{t}P_{s}
&=\sum_{i=1}^{n_{t}}\sum_{j=1}^{n_{s}}
d(t)d(s)\,\pi_{r}(\tilde{V}(t)_{i})^{*}
e_{0}\pi_{r}(\tilde{V}(t)_{i}\tilde{V}(s)_{j}^{*})e_{0}\pi_{r}(\tilde{V}(s)_{j})\\
&=\delta_{s,t}\sum_{i,j=1}^{n_{t}}
\delta_{i,j}d(t)\pi_{r}(\tilde{V}(t)_{i}^{*})
e_{0}\pi_{r}(\tilde{V}(t)_{j})
=\delta_{s,t}P_{t}.
\end{align*}
Thus $P_{t}=P_{t}^{*}$ is a projection and $P_{t}\perp P_{s}$ for all $s\neq t$.
Since the map $B\ni b\mapsto \eta(b)\in\cE$ is bounded and $B$-$B$ bilinear,
by Lemmata \ref{lem:orthogonal_summand} and \ref{lem:adjointable}, $\eta(B)$ is closed in $\cE$.
Hence $e_{0}$ is the projection onto
\[
\overline{e_{0}\eta(D^{\diamondsuit})}
=\overline{\eta(B)}
=\eta(B).
\]
This implies
\[
P_{t}\cE
\subset\mathrm{span}_{i}\,\pi_{r}(\tilde{V}(t)_{i})^{*}e_{0}\pi_{r}(\tilde{V}(t)_{i})\cE
\subset\mathrm{span}_{i}\,\pi_{r}(\tilde{V}(t)_{i})^{*}\eta(B)
=\eta(\cL_{t}^{*}B).
\]
Conversely, for any $x\in B$ and $1\le j\le n_{t}$, we have
\[
P_{t}\eta(\tilde{V}(t)_{j}^{*}x)
=\sum_{i=1}^{n_{t}}d(t)\pi_{r}(\tilde{V}(t)_{i}^{*})
e_{0}\pi_{r}(\tilde{V}(t)_{i}\tilde{V}(t)_{j}^{*})
e_{0}\eta(x)
=\eta(\tilde{V}(t)_{j}^{*}x).
\]
Therefore, $P_{t}$ is the projection onto $\eta(\cL_{t}^{*}B)$.
\end{proof}

Hereafter, we also denote by $\eta$ the bounded extension
$D_{r}\to\cE$ of $\eta$.
Indeed, for each $d\in D^{\diamondsuit}$, we have
\[
\|\eta(d)\|^{2}
=\|E_{0}(d^{*}d)\|
=\|E(\pi_{r}(d)^{*}\pi_{r}(d))\|
\le \|\pi_{r}(d)\|^{2}.
\]
Hence, $\eta\circ\pi_{r}^{-1}\colon \pi_{r}(D^{\diamondsuit})\to \cE$
extends to a bounded map on
$D_{r}=\overline{\pi_{r}(D^{\diamondsuit})}$,
which is again denoted by $\eta$.

The proof strategy for the following lemma is similar to that for \cite[Proposition 4.7]{doi:10.1142/S1793525325500256}.
\begin{lem}\label{lem:faithful}
    Under the above assumptions,
    the conditional expectation $E$ is faithful.
\end{lem}
\begin{proof}
    Take $d\in D_{r}$ with $de_{0}=0$.
    It suffices to show $d=0$.
    Let $P_{t}\in\End(\cE_{B})$ be the projection onto $\eta(\cL_{t}^{*}B)$ for each $t\in\cT$.
    Fixing $t_{1},t_{2}\in\cT$,
    we now show $P_{t_{1}}dP_{t_{2}}=0$.
    Take a sequence $(d_n)_{n}$ in $\pi_{r}(D^{\diamondsuit})$ with $\lim_{n\to\infty}d_{n}=d$.
     By the definition of $\cE$,
     we have \[
     \|\eta(d_{n})\|^{2}=\|\langle\eta(d_{n})\mid\eta(d_{n})\rangle\|=\|E(d_{n}^{*}d_{n})\|=\|E(d_{n}^{*}d_{n})e_{0}\|=\|e_{0}d_{n}^{*}d_{n}e_{0}\|.\]
    Since $de_{0}=0$,
    this implies $\lim_{n\to\infty}\|\eta(d_{n})\|^{2}=\lim_{n\to\infty}\|e_{0}d_{n}^{*}d_{n}e_{0}\|=0$.
    For each $s\in\cT$
    and $n$,
    set $d_{n}(s):=\eta^{-1}(P_{s}\eta(d_{n}))\in\pi_{r}(\cL_{s}^{*}B)$.
    Note that each $d_{n}$ can be written as the finite sum $\sum_{s\in\cT}d_{n}(s)$, and $\lim_{n\to\infty}\eta(d_{n}(s))=\lim_{n\to\infty}P_{s}\eta(d_{n})=0$ holds for every $s$.
    Since $\eta(\cL_{s}^{*}B)$ is closed in $\cE$,
    the restriction $\eta|_{\pi_{r}(\cL_{s}^{*}B)}\colon\pi_{r}(\cL_{s}^{*}B)\to \eta(\cL_{s}^{*}B)$
    of $\eta$ is a Banach space isomorphism. Hence its inverse is bounded, and $\lim_{n\to\infty}\eta(d_{n}(s))=0$ implies
 $\lim_{n\to\infty}d_{n}(s)=0$ for all $s$.
    Set \[F:=\{s\in \cT\mid(\rho_{t_{1}},\;\rho_{t_{2}}\circ\rho_{s})\neq0\}.\]
    Then we have $P_{t_{1}}d_{n}P_{t_{2}}=\sum_{s\in F}P_{t_{1}}d_{n}(s)P_{t_{2}}$.
    Indeed,
    as discussed in the proof of Lemma \ref{lem_alg}, 
    we get the following for every $s\in\cT$:
    \[\pi_{r}(\cL_{s}^{*}B)\eta(\cL_{t_{2}}^{*}B)\subset\eta(\cL_{s}^{*}\cL_{t_{2}}^{*}B)={\rm span}\bigg(\bigcup_{r\in\cT,\,(\rho_{r},\,\rho_{t_{2}}\circ\rho_{s})\neq0} \eta(\cL_{r}^{*}B)\bigg).\]
   Hence,
    $P_{t_{1}}d_{n}(s)P_{t_{2}}=0$ for all $s\notin F$.
    This implies
    $P_{t_{1}}d_{n}P_{t_{2}}=\sum_{s\in F}P_{t_{1}}d_{n}(s)P_{t_{2}}$.
    Since $F$ is a finite set,
    we get \[P_{t_{1}}dP_{t_{2}}=\lim_{n\to\infty}P_{t_{1}}d_{n}P_{t_{2}}=\sum_{s\in F}\lim_{n\to\infty}P_{t_{1}}d_{n}(s)P_{t_{2}}=0.\]
   This implies $d=0$, and we see that $E$ is faithful.
\end{proof}
Since $D^{\diamondsuit}$ is generated by the \Cs-algebra $B$ and isometries,
we can define the universal enveloping \Cs-algebra 
$$
D_{u}:=C^*_{u}(D^{\diamondsuit}).
$$
By construction,
the inclusion $B\subset D_{u}$ is nondegenerate.
The canonical quotient map from $D_{u}$ onto $D_{r}$ is also denoted
\[
q_r:D_u\to D_r.
\]
We shall now state and prove a technical lemma that will be useful in the following discussion.  

\begin{lem}\label{lem:max}
For any nonzero \Cs-algebra $C$ and any nonzero surjective $*$-homomorphism $\theta\colon D_{u}\to C$, there is a $*$-homomorphism $\overline{\theta}\colon  C\to D_{r}$ such that $q_{r}=\overline{\theta}\circ\theta$ holds.
\end{lem}
\begin{proof}
    It suffices to show ${\rm Ker}(\theta)\subset {\rm Ker}(q_{r})$.
    For the sake of contradiction,
    we suppose that there is a positive element $x\in {\rm Ker}(\theta)\setminus {\rm Ker}(q_{r})$.
    Since $E$ is faithful,
    we may assume $\|E\circ q_{r}(x)\|=1$.
    In the same way as in the proof of Lemma \ref{lem_alg} (2),
    for any $\epsilon>0$,
    we can take a positive norm-one element $b\in B$ satisfying
    \[bxb\approx_{\epsilon}bE\circ q_{r}(x)b\;\text{ and }\;\|bE\circ q_{r}(x)b\|>1-\epsilon.\]
    Fixing a positive element $b_{0}\in B$ of norm one, by pure infiniteness of $B$ together with simplicity, 
    we have $cbE\circ q_{r}(x)bc^{*}\approx_{\epsilon}b_{0}$ for some $c$ with $\|c\|=1$.
    Hence,
    we get
    \[cbxbc^{*}\approx_{2\epsilon} b_{0}.\]
    Since the left-hand side is contained in ${\rm Ker}(\theta)$,
    this implies $B\subset {\rm Ker}(\theta)$.
    Therefore ${\rm Ker}(\theta)=D_{u},$ contradicting $\theta\neq0$.
    This establishes the lemma. 
\end{proof}

\begin{proof}[Proof for Theorem \ref{thmalpha:end}]
Notice that $A\ncong\IK$.
Indeed, every irreducible endomorphism of $\IK$ is unitarily equivalent to $\id_{\IK}$, since it induces an irreducible representation of $\IK$, and every irreducible representation of $\IK$ is unitarily equivalent to the one associated with $\id_{\IK}$.
Hence, $\IK$ has no nontrivial irreducible sectors, whereas $A$ has an irreducible nonsurjective sector $[\rho]$.

    Let $q\colon B(\cH)\to B(\cH)/\IK(\cH)=:Q(\cH)$ be the canonical quotient map.
    Considering $C:=(\id_{\cO_{\infty}}\otimes q)(D)$,
    Inequality (\ref{eq:contradict}) implies the following inequality in $C$:
     \begin{align}\label{eq:quatient}
        \delta(1_{\cO_{\infty}}\otimes q\circ\pi (hh^{*}))&\leq (1_{\cO_{\infty}}\otimes q\circ\pi (h))\Big[(1_{\cO_{\infty}}\otimes q\circ\pi (a))(1_{\cO_{\infty}}\otimes q(V))\\
        &+(1_{\cO_{\infty}}\otimes q(V))^{*}(1_{\cO_{\infty}}\otimes q\circ\pi (a))^*\Big](1_{\cO_{\infty}}\otimes q\circ\pi (h^{*}))\nonumber
    \end{align}
for every $h\in H$.
    By the universality of $D_{u}$,
    the restriction $(\id_{\cO_{\infty}}\otimes q)|_{D^{\diamondsuit}}\colon D^{\diamondsuit}\to  C$ extends to a nonzero surjective $*$-homomorphism $\theta\colon D_{u}\to C$. 
    To see that $\theta$ is nonzero, it suffices to show that $q(\pi(A))\neq 0$, equivalently that $\pi(A)\not\subset \IK(\cH)$.
If $\pi(A)\subset \IK(\cH)$, then simplicity of $A$ implies $A\cong\pi(A)\cong \IK$
\footnote{Take a positive nonzero $a\in A\subset \IK$. Since $a$ has a finite spectrum away from $0$,
    there is a nonzero projection $p$ in $A$. As $p$ is a finite-rank projection in $\IK$, 
    $pAp$ is finite dimensional.
    Thus, we can take $q\in pAp$ with $qAq=\IC q$. Since $A$ is stable and simple,
    $A$ is isomorphic to the stabilization of $qAq$. 
    Thus, $A\cong\IK$.
    },
    contradicting the previous observation.
    Hence, $\theta$ is nonzero.
    Therefore, by the simplicity of $B$, the restriction $\theta|_{B}$ is an isomorphism between $B$ and $\cO_{\infty}\otimes q(\pi(A))$.
    
    By Lemma \ref{lem:max},
    there is a $*$-homomorphism $\overline{\theta}\colon C\to D_{r}$ satisfying $q_{r}=\overline{\theta}\circ \theta$,
    and we get the conditional expectation 
    \[E\circ\overline{\theta}\colon C\overset{\overline{\theta}}{\to} D_{r}\overset{E}{\to} B.\]
    Here, under the isomorphism $\theta|_{B}$, we identify $B$ with $\cO_{\infty}\otimes q\circ\pi (A)$ that is a subalgebra of $C$.
    Since the isometry $1_{\cO_{\infty}}\otimes V$ is contained in $\cL_{t_0}=(\iota\circ\rho_{t_{0}},\iota)$,
    we get 
    $$
    E\circ\overline{\theta}(1_{\cO_{\infty}}\otimes q(V))=E\circ q_r(1_{\cO_\infty}\otimes V) = E(1_{\cO_\infty}\otimes V)=0.
    $$
    Hence, Inequality (\ref{eq:quatient}) implies  
    \[0\geq E\circ\overline{\theta}[\delta(1_{\cO_{\infty}}\otimes q\circ\pi(hh^{*}))]=\delta(1_{\cO_{\infty}}\otimes q\circ\pi(hh^{*}))\]
    for every $h\in H$.
    Since $A$ is simple and $q\circ\pi$ is injective,
    the above inequality contradicts the assumption $H\neq0$.
    Thus $\rho$ is properly outer.
\end{proof}

\section{Applications}
\subsection{Freeness of UTC actions}
Freeness of actions of unitary tensor categories was introduced in \cite[Definition 4.1]{2024arXiv240918161H} to discuss the simplicity of the generalized crossed product \Cs-algebras associated with UTC-actions.
In this subsection, we show that outer actions of unitary tensor categories on unital simple \Cs-algebras are automatically free (see Corollary \ref{cor:csirr}).

First, we introduce proper outerness for bimodules. (In \cite[Definition 2.1]{MR4485960} this property is referred to as \emph{aperiodicity}.)
\begin{defn}\label{defn:bimod_prop_outer}
Let $X$ be a (right) \Cs-correspondence over a \Cs-algebra $A$.
We say that $X$ is {\bf properly outer} if for every $\xi\in X$,
every hereditary subalgebra $H\subset A$ and every $\epsilon>0$,
there exists a positive element $h\in H$ with $\|h\|=1$ such that $\|h\tr\xi\tl h\|<\epsilon$ holds. 
\end{defn}

The following lemma is repeatedly used throughout this section.
\begin{lem}\label{lem:functional_calculus}
    Let $A$ be a \Cs-algebra and $a\in A$ be a self-adjoint element with $\sup{\rm Spec}(a)=1$.
    Then for any $\epsilon>0$,
    there exists a self-adjoint element $a^{\prime}\approx_{\epsilon}a$ such that the following hereditary subalgebra $H_{a^{\prime}}\subset A$ is nonzero:
    \begin{equation}\label{eq:hereditary}
        H_{a^{\prime}}:=\{h\in A\mid ha^{\prime}=a^{\prime}h=h\}.
        \end{equation}
        Moreover, if $a$ is positive and has norm one, then $a^{\prime}$ can be chosen to be positive and of norm one as well.
\end{lem}
\begin{proof}
Before proving the claim,
note that $H_{a'}\subset A$ is a \Cs-subalgebra which moreover is hereditary since $h_1xh_2\in H_{a'}$ for each $x\in A$ and every $h_1, h_2\in H_{a'}.$
    Define the continuous function $f\colon (-\infty,1]\to(-\infty,1]$ by
    \[
    f(t):=\begin{cases}
        1& \text{ if $t>1-\epsilon$} \\
        (1-\epsilon)^{-1}t& \text{ if $0<t\leq1-\epsilon$}\\
        t& \text{ if $t\leq0$}
    \end{cases}
    \]
    and put $a^{\prime}:=f(a)$.
    Since $\sup_{t\leq 1}|t-f(t)|\leq\epsilon$, we have $\|a-a^{\prime}\|\leq\epsilon$.
    As $\sup{\rm Spec}(a)=1$, we can take a positive continuous function $g$ supported on $[1-\epsilon, 1]$ such that $\sup_{t\in{\rm Spec}(a)}|g(t)|=1$.
    This implies that $0\neq g(a)\in H_{a^{\prime}}$. 
    By construction, if $a$ is positive and norm one, then so is $a^{\prime}$.
\end{proof}

The following statement is an analogue of Lemma \ref{lem:prop_outer}, and its proof is similar to \cite[Lemma 3.2]{MR634163}. 
However, since the techniques in its proof are also useful throughout this section, we decided to include a brief proof for the reader's convenience.
\begin{lem}\label{lem:prop_outer_bimod}
    Let $A$ be a \Cs-algebra and $X_{1},X_{2},\dots,X_{n}$ be properly outer \Cs-correspondences over $A$.
    Then the direct sum $\bigoplus_{i=1}^{n}X_{i}$ is properly outer.
    More precisely, we get the following:
    \begin{enumerate}
    \item For any nonzero hereditary algebra $H\subset A$, $\xi_{i}\in X_{i}$ ($i=1,\dots,n$), 
    and $\epsilon>0$,
    there exists a positive element $c\in H$ with $\|c\|=1$ such that 
    \[\|c\rhd\xi_{i}\lhd c\|<\epsilon,\quad i=1,\dots,n.\]
        \item For any self-adjoint element $a\in A$, $\xi_{i}\in X_{i}$ ($i=1,\dots,n$), 
    and $\epsilon>0$,
    there exists a positive element $c\in A$ with $\|c\|=1$ such that 
    \[\|cac\|>\|a\|-\epsilon,\;\text{ and }\|c\rhd\xi_{i}\lhd c\|<\epsilon, \quad i=1,\dots,n.\]
    \end{enumerate}
\end{lem}
\begin{proof}
    We show (1).
    When $n=1$, the claim follows from the definition. Suppose that the claim holds for $n-1$.
    For $i=1,\dots,n$ and $\xi_{i}\in X_{i}$,
    by the assumption of induction,
    we can take a positive element $d\in H$ with $\|d\|=1$ satisfying $\|d\rhd\xi_{i}\lhd d\|<\epsilon$ for all $1\leq i\leq n-1$.
    Note that the inequalities $\|d\rhd\xi_{i}\lhd d\|<\epsilon$ are preserved by replacing $d$ with an element $d^{\prime}\in H_{+}$ that is close enough to $d$ in the norm.
    By Lemma \ref{lem:functional_calculus}, after replacing $d$ with a nearby element of $H_+$, 
    we may assume $H_{d}:=\{h\in H\mid hd=dh=h\}$ is nonzero.
    Note that $H_{d}$ is a hereditary subalgebra of both $H$ and $A$.
    Since $X_{n}$ is properly outer,
    we can take a positive norm-one element $c\in H_{d}\subset H$ satisfying $\|c\rhd\xi_{n}\lhd c\|<\epsilon$.
    For each $1\leq i\leq n-1$,
    we have \[\|c\rhd\xi_{i}\lhd c\|=\|cd\rhd\xi_{i}\lhd dc\|=\|c\rhd(d\rhd\xi_{i}\lhd d)\lhd c\|<\epsilon.\] 
    Hence,
    we get (1) by induction.  Proper outerness of $\bigoplus_{i=1}^{n}X_{i}$ follows immediately from (1).
    
    For (2),
    replacing $a$ with $\|a\|^{-1}a$ or $-\|a\|^{-1}a$,
    we may assume 
    $\|a\|=\sup \Spec(a)=1$.
    Using Lemma \ref{lem:functional_calculus},
    we can take $a^{\prime}\approx_{\epsilon}a$ with $H_{a^{\prime}}\neq 0$.
    Applying (1) with $H=H_{a^{\prime}}$,
    we get a positive norm-one element $c\in H_{a^{\prime}}$ satisfying $\|c\rhd\xi_{i}\lhd c\|<\epsilon$ for all $1\leq i\leq n$.
    Since we have
    \[\|cac\|>\|ca^{\prime}c\|-\epsilon=\|c^{2}\|-\epsilon=1-\epsilon,\]
    this completes the proof.
\end{proof}

\begin{rmk}\label{rmk:outer}
    If $X$ is properly outer,
    then $\rCorr({}_{A}A_{A}\to {}_{A}X_{A})=0$ holds.
    Indeed,
    if there exists $T\in\rCorr({}_{A}A_{A}\to {}_{A}X_{A})$ with $\|T\|=1$,
    then for any $\epsilon>0$,
    we can take a hereditary subalgebra $H\subset A$ such that $T^{*}Th\approx_{\epsilon}h$ for all $h\in H$ with $\|h\|=1$.
Indeed,
 by Lemma \ref{lem:functional_calculus} applied to
$a=T^{*}T$,
there exists a positive element $S\in M(A)$ of norm one such that $S\approx_{\epsilon}T^{*}T$ and the hereditary subalgebra $H_S\subset M(A)$ defined in Equation (\ref{eq:hereditary}) is nonzero.
It follows that $H:=\overline{H_{S}AH_{S}}$ is a nonzero hereditary subalgebra of $A$ satisfying the required property.\\
Using Lemma \ref{lem:functional_calculus} repeatedly,
we choose a positive norm-one element $h_0\in H$ such that the hereditary subalgebra
$
H_1:=H_{h_{0}}:=\{h\in H\mid hh_0=h=h_0h\}\subset H
$
is nonzero.
    We then have
    \[\|h_1\rhd Th_0\lhd h_1\|^{2}=\|T(h_1h_0h_1)\|^{2}=\|T(h_1^{2})\|^{2}=\|h_{1}^{2}(T^{*}T)h_{1}^{2}\|\approx_{\epsilon}\|h_1^{4}\|\]
    for every positive norm-one element $h_1\in H_{1}$.
    Note that $H_{1}$ is also hereditary in $A$.
    This implies that $X$ is not properly outer.
\end{rmk}

 When $A$ is separable and simple,
    we get the converse as follows.
\begin{cor}\label{cor:properly_outer}
    Let $A$ be a separable simple \Cs-algebra,
    then every finite-index bimodule $X$ over $A$ with $\rCorr({}_{A}A_{A}\to {}_{A}X_{A})=0$ is properly outer.
\end{cor}
\begin{proof}
   First,
    we assume $X$ is irreducible.
    Let $X_{s}:=\IK\otimes X$ be an exterior tensor product,
    then it is routine to show that $X_{s}$ is a finite-index \Cs-bimodule over $A_{s}:=\IK\otimes A$.
    Thanks to Lemma \ref{lem:stable_bimodule},
    we get a finite-index endomorphism $\rho\colon A_{s}\to A_{s}$ such that ${}_{A_{s}}{X_{s}}_{A{s}}$ is isomorphic to ${}_{A_{s}}({}_{\rho}A_{s})_{A{s}}$.
   Since ${}_{A_{s}}{X_{s}}_{A{s}}$ is irreducible,
   $\rho$ is irreducible.
The assumption $\rCorr({}_{A_{s}}A_{sA_{s}}\to {}_{A_{s}}X_{sA_{s}})=0$ implies $(\rho,\id_{A_{s}})=0$.
Combining Theorem \ref{thmalpha:end} with \cite[Lemma 1.1]{MR634163},
we see that $\rho$ is properly outer regardless of whether it is surjective.
Hence,
   we get $X_{s}\cong {}_{\rho}A_{s}$ is properly outer.

    To show the proper outerness of $X$,
    fix a hereditary subalgebra $H\subset A$,
    $\epsilon>0$, 
    a norm one vector $\xi\in X$
    and a rank one projection $e_{1,1}\in\IK$.
    Since $e_{1,1}\otimes H$ is a hereditary subalgebra of $A_{s}$ and $e_{1,1}\otimes\xi\in X_{s}$,
    there is a norm-one element $c\in H$ with $\|c^{*}\rhd\xi\lhd c\|_{X}=\|e_{1,1}\otimes(c^{*}\rhd\xi\lhd c)\|_{X_{s}}<\epsilon$.
    This shows the claim for irreducible $X$.

    Whenever $X$ is decomposed into a direct sum $\bigoplus_{i=1}^{n}X_{i}$ of irreducible finite-index bimodules (see Lemma \ref{lem:semisimple_bimodule}), since every direct summand $X_{i}$ satisfies $\rCorr({}_{A}A_{A}\to {}_{A}X_{iA})=0$,
    it is properly outer.
    Thanks to Lemma \ref{lem:prop_outer_bimod}, 
    we get the statement.
\end{proof}

We now explore the interplay between proper outerness and the notion of \emph{freeness} of UTC-actions in the sense of \cite{2024arXiv240918161H}, which we recall below. 
In this subsection, we write objects of $\cC$ as $\ca,\cb$.
\begin{defn}[Definition 4.1 of \cite{2024arXiv240918161H}]\label{defn:free_action}
    Let $\cC$ be a unitary tensor category, and let $A$ be a unital \Cs-algebra.
    We say that an outer action (i.e. a fully-faithful unitary tensor functor) $F\colon \cC\to \fiBim(A)$ {\bf is free }if for each $\ca\in \cC$ with $\cC(1_{\cC}\to \ca)=0$ and every $\xi\in F(\ca)$ we have
    \[
    \inf
    \left\{\Big\|\sum_{i=1}^{n}s_{i}\rhd\xi\lhd s_{i}^{*}\Big\|\, \  
    \middle| 
    \ s_{1},\dots,s_{n}\in A,\;\sum_{i=1}^{n}s_{i}s_{i}^{*} =1\right\}=0.
    \]
\end{defn}
\begin{rmk}
    Every free UTC action $F\colon \cC\to\fiBim(A)$ on a unital \Cs-algebra $A$ with trivial center is automatically outer.
    Indeed,
    if $\cb\in\cC$ is an irreducible object,
    then $\cb\otimes\overline{\cb}$ is decomposed as a direct sum $1_{\cC}\oplus \ca$ such that $\cC(1_{\cC}\to \ca)=0$.
    Hence,
    we have
    \[F(\cb)\boxtimes_{A}\overline{F(\cb)}\cong F(\cb\otimes\overline{\cb})\cong A\oplus F(\ca).\]
    Since $F$ is free,
    $\rCorr({}_{A}A_{A}\to{}_{A}F(\alpha)_{A})=0$ holds.
    Indeed, if there is an isometry $T\in\rCorr({}_{A}A_{A}\to {}_{A}F(\ca)_{A})$,
    then for any $s_{1},\dots,s_{n}\in A$ with $\sum_{i=1}^{n}s_{i}s_{i}^{*}=1_{A}$ we have $\sum_{i=1}^{n}s_{i}\rhd T(1_{A})\lhd s_{i}^{*}=T(1_{A})$.
    This contradicts the freeness of $F$.
    Consequently,
    since $F(\cb)\boxtimes_{A}\overline{F(\cb)}$ contains the trivial bimodule as an orthogonal summand with multiplicity 1,
    we see that $F(\cb)$ is an irreducible bimodule for every irreducible object $\cb\in\cC$.
    This implies that $F$ is outer.
\end{rmk}
When a \Cs-algebra is simple,
the converse holds as follows.

\begin{cor}\label{cor:free}
    When a \Cs-algebra $A$ is unital, 
    separable and simple,
    every outer action $F\colon \cC\to \fiBim(A)$ of a unitary tensor category $\cC$ is free.
\end{cor}
\begin{proof}
    For every $\ca\in\cC$ with $\cC(1_{\cC}\to \ca)=0$,
    the finitely generated projective bimodule $F(\ca)$ satisfies $\rCorr({}_{A}A_{A}\to{}_{A}F(\ca)_{A})=0$.
    Fix $\xi\in F(\ca)$ and $\epsilon>0$.
   By Lemma \ref{lem:prop_outer_bimod},
    we can take a positive norm-one element $c_{0}\in A$ such that $\|c_{0}\triangleright\xi\triangleleft c_{0}\|<\epsilon$.
  Applying Lemma \ref{lem:functional_calculus} to $a=c_{0}$ and replacing $c_{0}$ if necessary,
   we may assume there exists a positive norm-one element $c\in A$ with $cc_{0}=c=c_{0}c$.
   Since $A$ is unital and simple,
   we get $t_{1},\dots,t_{n}\in A$ with $\sum_{j=1}^{n}t_{j}c^{2}t_{j}^{*}=1_{A}$.
   Let $\ell^{2}([n])$ be the $n$-dimensional Hilbert space with orthonormal basis $\{\delta_{i}\}_{i=1}^{n}$. 
   Define 
   \begin{align*}
   s_{j}&:=t_{j}c,\qquad\qquad\qquad \zeta:=\sum_{j=1}^{n}(c_{0}\rhd\xi\lhd c_{0}s_{j}^{*})\otimes\delta_{j}\in F(\ca)\otimes\ell^{2}([n]),\\
   x&:=\sum_{i,j=1}^{n}s_{i}^{*}s_{j}\otimes e_{i,j},\quad s:=\sum_{j=1}^{n}s_{j}\otimes e_{1,j}\in A\otimes M_{n}\subset \End\left((F(\ca)\otimes\ell^{2}([n]))_{A}\right).
   \end{align*}
   Note that 
   \[
   \|x\|=\|s^{*}s\|=\|ss^{*}\| = \bigg\|\sum_{j=1}^{n}s_{j}s_{j}^{*}\otimes e_{1,1}\bigg\| = \bigg\|\sum_{j=1}^{n}s_{j}s_{j}^{*}\bigg\| = \bigg\|\sum_{j=1}^{n}t_{j}c^{2}t_{j}^{*}\bigg\| = 1.
   \]
   Since $s_{j}c_{0}=s_{j}$, we have
   \begin{align*}
\bigg\|\sum_{j=1}^{n}s_{j}\triangleright\xi\triangleleft s_{j}^{*}\bigg\|^{2}
=&\bigg\|\sum_{i,j=1}^{n}s_{i}\langle c_{0}\triangleright\xi\triangleleft c_{0}\mid(s_{i}^{*}s_{j}c_{0})\triangleright\xi\triangleleft c_{0}\rangle s_{j}^{*}\bigg\|\\
=&\|\langle\zeta\mid x\rhd\zeta\rangle\|\\
\leq&\|x\|\cdot\|\langle\zeta\mid\zeta\rangle\|\\
=& 
\bigg\|\sum_{j=1}^{n}s_{j}\langle c_{0}\triangleright\xi\triangleleft c_{0}\mid c_{0}\triangleright\xi\triangleleft c_{0}\rangle s_{j}^{*}\bigg\|\\
\leq&\bigg\|\sum_{j=1}^{n}s_{j}s_{j}^{*}\bigg\|\cdot \|\langle c_{0}\triangleright\xi\triangleleft c_{0}\mid c_{0}\triangleright\xi\triangleleft c_{0}\rangle\|\\
<&\epsilon^{2}.
   \end{align*}
   Thus,
   we get $\|\sum_{j}^{n}s_{j}\triangleright\xi\triangleleft s_{j}^{*}\|<\epsilon$.
\end{proof}
The above statement implies that the freeness assumption in \cite[Theorem 4.2]{2024arXiv240918161H} automatically holds if an action is outer. 

\begin{rmk}
    In the proof of Corollary \ref{cor:free} we used that for a unital simple \Cs-algebra $A$, and any positive $c\in A$, there are $\{t_j\}_{j=1}^n\subset A$ such that $\sum_{j=1}^n t_jct_j^* =1$. 
    We recall the proof of this elementary fact here for the reader's convenience. 
    Since $A$ is unital and simple, we have $\sum_{j=1}^n t_j\sqrt{c}t_j^{\prime}=1$ for some $\{t_j, t_j'\}_{j=1}^n\subset A$.
    Let us assume first $n=1$. Then $t\sqrt{c}t^{\prime}=1$.
   We have
   \[\|t^{\prime}t^{\prime*}\|tct^{*}\geq t\sqrt{c}t^{\prime}t^{\prime*}\sqrt{c}t^{*}=1.\]
   This implies $tct^{*}$ is positive invertible. Hence, we have $1=(tct^{*})^{-\frac{1}{2}}tct^{*}(tct^{*})^{-\frac{1}{2}}$. Replacing $t$ with $(tct^{*})^{-\frac{1}{2}}t$, we have $tct^{*}=1$.
   The general fact follows at once from a standard matrix trick. 
\end{rmk}

\subsection{\texorpdfstring{\Cs-discrete inclusions of simple \Cs-algebras}{}}\label{sec:disc}
Proper outerness of automorphisms is useful for the study of crossed product \Cs-algebras.
In this subsection,
using proper outerness of bimodules,
we show several properties of \Cs-discrete inclusions.
We focus on \Cs-irreducibility and some properties of inclusions arising from quasi-product actions \cite{MR1234394} of compact groups.

    We use the following definition for a nondegenerate inclusion that is not necessarily unital.
\begin{defn}[Definition 3.1 of \cite{MR4599249}]\label{defn:Csirrd}
    A nondegenerate inclusion $B\subset A$ of \Cs-algebras is said to be {\bf \Cs-irreducible} if every intermediate \Cs-algebra is simple.
\end{defn}
In \cite[Theorem 5.8]{MR4599249},
it is proven that if a discrete group action $\Gamma\acts B$ on a unital simple \Cs-algebra is outer,
then the inclusion $B\subset B\rtimes_{r}\Gamma$ is \Cs-irreducible.
Here, we note that the proof of \cite[Theorem 5.8 (ii)$\Rightarrow$ (i)]{MR4599249} also works for actions on nonunital \Cs-algebras. Since we will prove \Cs-irreducibility in a more general setting than
$B\subset B\rtimes_r\Gamma$ in Corollary \ref{cor:csirr},
we omit the details.

    As a dual notion to a properly outer action of a discrete group,
    we may consider a quasi-product action of a compact group introduced in \cite{MR1234394}.
    \begin{thm}[Theorem 1 of \cite{MR1234394}]\label{thm:qpa}
        Let $\alpha\colon G\acts A$ be a faithful compact group action on a separable \Cs-algebra,
        then the following are equivalent:
        \begin{enumerate}
            \item The crossed product \Cs-algebra $A\rtimes G$ is prime,
            and all nontrivial dual endomorphisms $\{\hat{\alpha}_{\pi}\}_{\pi\in\hat{G}\setminus\{1\}}$ of $(A\rtimes G)\otimes\IK$, (see \cite[Section 1]{MR1234394} for the definition), are properly outer.
            \item There is a $\alpha$-invariant pure state $\omega$ of $A$ such that the restriction $\omega|_{A^{G}}$ defines a faithful GNS-representation $\pi_{\omega|_{A^{G}}}$ of $A^{G}$.
            \item There is a faithful irreducible representation $\pi\colon A\to B(\cH)$ such that the restriction $\pi|_{A^{G}}$ is also irreducible.
        \end{enumerate}
    \end{thm}
    If an action $\alpha$ satisfies the above equivalent conditions,
    $\alpha$ is called a {\bf quasi-product action}.
    The last item of the above theorem describes a property concerning the irreducibility of the inclusion $A^{G}\subset A$. This is referred to as the {\bf property (BEK)} as follows.
    \begin{defn}[Definition 2.4 of \cite{doi:10.1142/S0129055X24610026} (see also Theorem 3.1 of \cite{MR1234394})]
        We say that an inclusion $B\subset A$ of separable \Cs-algebras {\bf has the (BEK) property},
        if there is a faithful irreducible representation $\pi\colon A\to B(\cH)$ such that the restriction $\pi|_{B}$ is also irreducible.
    \end{defn}
Recently,
Izumi proved that if a compact group action $\alpha\colon G\acts A$ on a separable \Cs-algebra is minimal and the fixed point algebra $A^{G}$ is simple,
then $\alpha$ is a quasi-product action \cite[Theorem 1.1]{doi:10.1142/S0129055X24610026}.
In this subsection,
we study irreducible \Cs-discrete inclusions of simple \Cs-algebras by comparing them with inclusions arising from quasi-product actions of compact groups and properly outer actions of discrete groups.
First, we prepare the following lemma.
For the definition of the diamond space, refer to Equation (\ref{eq:diamond}); for the definition of the projective quasi-normalizer, refer to Equation (\ref{eqn:PQN}).

\begin{lem}\label{lem:PQN}
    Let $B\overset{E}{\subset}A$ be a nondegenerate irreducible inclusion with a faithful conditional expectation $E$.
    As in Example \ref{exa:basic_construction},
    we suppose $\cE$ is a \Cs-correspondence associated with $E$ and $\eta\colon A\to \cE$ is the canonical inclusion map, 
    then the following hold.
    \begin{enumerate}
    \item The diamond space $\rCorr({}_{B}B_{B}\to{}_{B}\cE_{B})^{\diamondsuit}$ is equal to $\IC\eta \circ\iota$,
    where $\iota\colon B\to A$ is the given inclusion map.
    \item For every $a\in \PQN(B\subset A)$,
    there exist $X\in\fiBim(B)$ and $f\in\rCorr({}_{B}X_{B}\to{}_{B}\cE_{B})^{\diamondsuit}$ such that $\rCorr({}_{B}B_{B}\to{}_{B}X_{B})=0$ and $\eta(a-E(a))\in f(X)$.
    \item If $B$ is separable and simple,
    then for any $a_{1}, a_{2},\dots,a_{n}\in \PQN(B\subset A)$, any $\epsilon>0$, and each nonzero hereditary subalgebra $H\subset B$,
    there exists an element $b\in H_{+}$ such that
    \[\|b\|=1,\qquad \|b(a_{i}-E(a_{i}))b\|<\epsilon.\]
    \end{enumerate}
\end{lem}
\begin{proof}
    For (1),
    take $f\in\rCorr(B\to \cE)^{\diamondsuit}$.
    We put $\check{f}:=\eta^{-1}\circ f\colon B\to A$, and we shall show it is a continuous $B$-$B$ bilinear map.
    By Lemma \ref{lem:orthogonal_summand}, 
    we have $f(B)$ is closed in $\cE$,
    and $\eta|\colon \check{f}(B)\to f(B)$ is a continuous open linear bijection by the open mapping theorem.
    Here,
    $\check{f}(B)\subset A$ is closed in the norm since $\eta$ is norm continuous and $\check{f}(B)=\eta^{-1}(f(B))$.
    Hence, $\check{f}\colon B\to A$ is a bounded $B$-$B$ bilinear map.
    Let $(b_{n})_{n}$ be approximate units of $B$.
    Since $B\subset A$ is nondegenerate,
    straightforward computations show the sequence $(\check{f}(b_{n}))_{n}$ converges in the strict topology and the limit $\lim_{n}\check{f}(b_{n})$ is in $M(A)\cap B^{\prime}\cong\IC$.
    Let $\lim_{n}\check{f}(b_{n})=\lambda1_{A}$,
    then we have \[f(b)=\lim_{n\to\infty}f(b_{n}b)=\lim_{n\to\infty}\eta(\check{f}(b_{n})b)=\lambda\eta(b)\quad\text{for all $b\in B$.}\]

    For (2), 
    take $a\in \PQN(B\subset A)$ arbitrarily.
    By the definition of the projective quasi-normalizer,
    there exist
    $Y\in\fiBim(B)$ and $g\in\rCorr({}_{B}Y_{B}\to{}_{B}\cE_{B})^{\diamondsuit}$ with $a\in\check{g}(Y)$.
    Let $e$ be the Jones projection,
    then $h:=(1-e)g$ is also an element of $\rCorr({}_{B}Y_{B}\to{}_{B}\cE_{B})^{\diamondsuit}$.
    Put $X:=h^{*}h(Y)$,
    and the restriction $h|_{X}\colon X\to \cE$ of $h$ is denoted by $f$.
    We will show that these $f$ and $X$ satisfy the required property.
    As in the proof of Lemma \ref{lem:orthogonal_summand},
    we get $X=h^{*}h(Y)=\ker(h)^{\perp}$ and 
    \[f(X)=hh^{*}h(Y)=h(Y)=(1-e)(g(Y))\subset \eta(A).\]
    Hence,
    we see that $f\in\rCorr({}_{B}X_{B}\to {}_{B}\cE_{B})^{\diamondsuit}$ and $\eta(a-E(a))\in f(X)$.
It is enough to show $\rCorr({}_{B}B_{B}\to{}_{B}X_{B})=0$.
Since $f^{*}f\in\End({}_{B}X_{B})$ is invertible,
    for each $k\in\rCorr({}_{B}B_{B}\to {}_{B}X_{B})$,
   we have $k=(f^{*}f)^{-1}f^{*}fk$.
   On the other hand,
   using (1),
   we see that $fk\in\rCorr({}_{B}B_{B}\to{}_{B}\cE_{B})^{\diamondsuit}=\IC\eta\circ\iota$ and $efk=0$.
   This implies $fk=0$.
   Thus, we get $\rCorr({}_{B}B_{B}\to {}_{B}X_{B})=0$.
\smallskip

   We prove (3).
   By (2), for each $i$,
we can take $X_{i}\in\fiBim(B)$ with $\rCorr({}_{B}B_{B}\to{}_{B}X_{iB})=0$
and $f_{i}\in\rCorr({}_{B}X_{iB}\to{}_{B}\cE_{B})^{\diamondsuit}$ satisfying $a_{i}-E(a_{i})\in\check{f}_{i}(X_{i})$.
Take $\xi_{i}\in X_{i}$ such that $\check{f}_{i}(\xi_{i})=a_{i}-E(a_{i})$. 
Applying Corollary \ref{cor:properly_outer} and Lemma \ref{lem:prop_outer_bimod} (1), 
for any $\epsilon>0$ and any nonzero hereditary \Cs-subalgebra $H\subset B$, 
there exists $b\in H_{+}$ such that
    \[\|b\|=1,\quad 
    \|b\rhd\xi_{i}\lhd b\|<\frac{\epsilon}{\max_{1\leq i\leq n}\|\check{f}_{i}\|+1}.\]
Here, in the same way as in the proof of (1),
we see that $\check{f}_{i}\colon X_{i}\to A$ is bounded.
Therefore,
we have 
\[\|b(a_{i}-E(a_{i}))b\|=\|b\check{f}_{i}(\xi_{i})b\|=\|\check{f}_{i}(b\rhd\xi_{i}\lhd b)\|<\epsilon.\]
\end{proof}

Using Lemma \ref{lem:PQN} and Theorem \ref{thm:end},
we obtain some structural properties of irreducible \Cs-discrete inclusions (c.f. Definition \ref{defn:C*Disc}). 
\begin{cor}\label{cor:csirr}
    Let $A$ and $B$ be \Cs-algebras, and 
    let $B\overset{E}{\subset}A$ be an irreducible \Cs-discrete inclusion.
    If $B$ is separable and simple,
    then the following hold.
    \begin{enumerate}
        \item When $A$ is separable,
        there is a pure state $\omega\colon A\to \IC$ with $\omega=\omega\circ E$.
        \item The inclusion $B\subset A$ is \Cs-irreducible. In particular, $A$ is automatically simple. 
        \item If $B$ is purely infinite,
        then every intermediate \Cs-algebra is purely infinite.
        \item $E$ is the unique conditional expectation from $A$ onto $B$.
        \item If $E$ is of finite index,
        then $B\subset A$ has the property (BEK).
    \end{enumerate}
\end{cor}

\begin{proof}
We show (1).
 The required pure state is obtained by the same argument as in the proofs of \cite[Lemma 3.2]{MR634163} and \cite[Theorem 3.1]{MR1234394}, but we recall the proof for the reader's convenience since our setting is a little bit different from that in the references.
Since $A$ is separable and the inclusion $B\subset A$ is \Cs-discrete,
we can take a sequence $\{a_{n}\}_{n}$ in $\PQN(B\subset A)$, which is dense in $A$.
Using Lemma \ref{lem:functional_calculus} and Lemma \ref{lem:PQN} (3) repeatedly,
    we get a sequence $(b_{N})_{N}$ of $B_{+}$ satisfying 
    \[\|b_{N}\|=1,\ 
    b_{N}b_{N+1}=b_{N+1},\quad\text{ and }\quad\ 
    \|b_{N}(a_{n}-E(a_{n}))b_{N}\|<\frac{1}{N}\]
    for all $N$ and $1\leq n\leq N$.
    Indeed,
    let $H_{1}:=B$.
    By Lemma \ref{lem:PQN} (3),
    we can take $b_{1}\in H_{1+}$ with $\|b_{1}\|=1$ and $\|b_{1}(a_{1}-E(a_{1}))b_{1}\|<1$.
    By Lemma \ref{lem:functional_calculus}, after replacing $b_1$ if necessary,
    we may assume that $H_{2}:=\{b\in H_{1}\mid bb_{1}=b_{1}b=b\}\subset H_{1}$ is nonzero.
    Using the same argument repeatedly,
    we get $b_{2}\in H_{2+}$ satisfying $\|b_{2}\|=1$, $\max_{i\leq 2}\|b_{2}(a_{i}-E(a_{i}))b_{2}\|<\frac{1}{2}$, and $H_{3}:=\{b\in H_{2}\mid bb_{2}=b_{2}b=b\}\neq0$.
    Similarly, we get the required sequence $(b_{N})_{N}$ by induction.
    Hence,
    we get $\lim_{N\to\infty}\|b_{N}(a_{n}-E(a_{n}))b_{N}\|=0
    $
    for all $n$. 
   Since $\PQN(B\subset A)$ is dense in $A$,
    this implies 
    \begin{equation}\label{eq:relative_ex}
    \lim_{N\to\infty}\|b_{N}(a-E(a))b_{N}\|=0
    \end{equation}
    for every $a\in A$.
    Set $K_{N}:=\{\varphi\mid\varphi \ \text{ is a state of $A$ with }\ \varphi(b_N^{2})=1\}$.
    Since $(K_N)_{N}$ is a decreasing sequence of nonempty faces of the state space of $A$,
    we get a pure state $\omega$ in the set of extreme points of $\cap_{N}K_{N}$.
    For each $N$,
    since we have $1\geq\omega(b_{N})\geq\omega(b_{N}^{2})=1$ by the definition of $K_{N}$,
    $b_{N}$ belongs to the multiplicative domain of $\omega$.
    Thus,
    we get 
    \[
    \omega(a)=\omega(b_{N}ab_{N})
    \]
     for all $N$ and $a\in A$.
    Thanks to (\ref{eq:relative_ex}),
    we have $\omega=\omega\circ E$.
\smallskip

    We show (2).
    Let $D$ be an intermediate \Cs-algebra between $B$ and $A$, and let $I$ be a nonzero ideal of $D$.
    It is enough to show $D=I$.
    Take positive elements $d\in I$ and $b\in B$ with $\|E(d)\|=\|b\|=1$.
    By Lemma \ref{lem:PQN} (3),
    for each $\epsilon>0$,
    we can take a positive element $c\in B$ such that
$cdc\approx_{\epsilon}cE(d)c\;\text{ and }\|cE(d)c\|=1$
    hold.
    Using \cite[Lemma A.2]{MR4377313},
    we get $x_{1},\dots,x_{n}\in B$ satisfying
    \begin{equation}\label{eq:approx}\sum_{i=1}^{n}x_{i}cdcx_{i}^{*}\approx_{\epsilon}\sum_{i=1}^{n}x_{i}cE(d)cx_{i}^{*}\approx_{\epsilon}b.\end{equation}
    Since the left-hand side belongs to $I$, and $B$ is simple,
    we see that $I\cap B=B$.
    Note that the inclusion $B\subset D$ is nondegenerate.
Indeed, since $B\subset A$ is nondegenerate, there exists an approximate unit for $A$ contained in $B$, and hence $B\subset D$ is nondegenerate.
Thus, we get $I=D$.

    We show (3). 
    Suppose $B$ is purely infinite and simple,
    then $B$ is either unital or stable \cite[Theorem 1.2]{MR1174483}.
    When $B$ is stable, 
    there is a projection $p\in B$, and we get an isomorphism of inclusions
    \[pBp\otimes\IK\overset{E|_{pAp}\otimes\id_{\IK}}{\subset}pAp\otimes\IK\cong B\overset{E}{\subset}A.\]
    Indeed, since $B$ is purely infinite simple, taking a nonzero projection $q\in B$, 
    we have $B\cong B\otimes\IK\cong qBq\otimes\IK$ by Brown's theorem \cite{MR454645}.
    Thus,
    we obtain nondegenerate inclusions $\IK\subset B\subset A$.
    That is, the inclusion $\IK\cong q\otimes\IK\subset qBq\otimes\IK\cong B$ is nondegenerate.
    Let $\{e_{i,j}\}_{i,j}$ be matrix units of $\IK$ and set $p:=e_{1,1}$.
    Then the following $*$-homomorphisms are mutually inverse
    \begin{align*}
        A&\overset{\theta}{\longrightarrow} pAp\otimes\IK,\qquad\qquad\qquad pAp\otimes\IK\overset{\theta^{-1}}{\longrightarrow}A\\
        a&\mapsto \sum_{i,j}e_{1,i}ae_{j,1}\otimes e_{i,j},\qquad\quad
        a\otimes e_{i,j}\mapsto e_{i,1}ae_{1,j}.
    \end{align*}
    The restriction of $\theta$ gives an isomorphism between $B$ and $pBp\otimes\IK$.
    Using $\theta^{-1}$,
    for any $a\otimes e_{i,j}\in pAp\otimes\IK$,
    we get 
   \[E(\theta^{-1}(a\otimes e_{i,j}))=E(e_{i,1}ae_{1,j})=e_{i,1}E(a)e_{1,j}=\theta^{-1}\left((E|_{pAp}\otimes\id_{\IK})(a\otimes e_{i,j})\right).
   \]
   Here, we use $e_{i,j}\in\IK\subset B$ for the second equality.
   Thus, we have the required isomorphism of inclusions preserving the conditional expectations.
    
    Replacing $B\subset A$ with $pBp\subset pAp$,
    we may assume $B$ is unital.
    Take an intermediate subalgebra $D$ between $B$ and $A$ and a positive nonzero element $d\in D$ with $\|E(d)\|=1$.
    Considering the approximation (\ref{eq:approx}) for $b=1$,
    we may assume $n=1$ since $B$ is purely infinite simple.
    Hence,
    we get $x_{1}cdcx_{1}^{*}\approx_{2\epsilon}1$.
    This implies that for each $d\in D_{+}\setminus\{0\}$ there exists $y\in D$ with $ydy^{*}=1$. Hence, $D$ is purely infinite simple.
\smallskip

    For (4),
    let $F\colon A\to B$ be a conditional expectation.
    If $F\neq E$,
    then there is a self-adjoint element $a\in A$ such that $\|F(a-E(a))\|=1$.
    Since $B\subset A$ is \Cs-discrete,
    we may assume that $a\in \PQN(B\subset A)$. Replacing $a$ with $-a$ if necessary, we have that $\sup{\rm Spec}(F(a-E(a)))=1$. 
    By Lemma \ref{lem:functional_calculus},
    for any $\epsilon>0$,
    we get a self-adjoint element $b^{\prime}\approx_{\epsilon}F(a-E(a))$ such that the hereditary subalgebra $H_{b^{\prime}}:=\{x\in B\mid xb^{\prime}=b^{\prime}x=x\}\subset B$ is nonzero. 
    By Lemma \ref{lem:PQN} (3), we get $b\in H_{b^{\prime}+}$ with $\|b\|=1$ such that $\|b(a-E(a))b\|<\epsilon$.
    It follows that
    \[1=\|bb^{\prime}b\|\approx_{\epsilon}\|b(F(a-E(a)))b\|=\|F(b(a-E(a))b)\|<\epsilon,\] which is absurd. Therefore, we get $E=F$.

   For (5),
    since $A=\PQN(B\subset A)$,
    the statement follows from Lemma \ref{lem:PQN} (3) and \cite[Theorem 3.5 (3)]{doi:10.1142/S0129055X24610026}.
\end{proof}
\begin{rmk}
\begin{itemize}
    \item 
    Corollary \ref{cor:csirr} (1) is an analogue of Theorem \ref{thm:qpa} (3).
    \item If the inclusion $B\subset A$ is irreducible and of finite index,
    the statement in Corollary \ref{cor:csirr} (4) is well-known and follows easily from \cite[Theorem 2.12.3]{MR996807}.
    \item
    By definition, finite-index inclusions of \Cs-algebras are \Cs-discrete.
    Corollary \ref{cor:csirr} (5) implies that the equivalent conditions in \cite[Theorem 3.5]{doi:10.1142/S0129055X24610026} automatically hold.
    \end{itemize}
\end{rmk}

\appendix

\section{Appendix}

In \cite{doi:10.1142/S0129055X24610026}, Corollary 3.4 and Theorem 3.1 are proved by using \cite[Theorem 3.1]{MR1234394}, and we noticed some potential gaps in the proof of \cite[Theorem 3.1]{MR1234394}. 
    For the reader's convenience, we shall provide here a revised argument due to Izumi.

\begin{thm}[Theorem 3.1 of \cite{MR1234394}]
    Let $B\subset A$ be an inclusion of separable \Cs-algebras,
    then the following are equivalent.
    \begin{enumerate}
        \item The following holds for every $x,y\in A$:\[\sup\{\|xby\|\;\mid b\in B, \|b\|=1\}=\|x\|\|y\|.\]
        \item There exists $\delta>0$ satisfying
        \[\sup\{\|xby\|\;\mid b\in B, \|b\|=1\}\geq\delta\|x\|\|y\|.\]
        \item There exists a faithful irreducible representation $\pi$ of $A$ such that the restriction $\pi|_{B}$ is also irreducible.
    \end{enumerate}
\end{thm}
\begin{proof}
    The implication $(1)\Rightarrow(2)$ is trivial, and $(3)\Rightarrow(1)$ follows from Kadison transitivity.

    To get $(2)\Rightarrow(3)$,
    we prepare a sequence $\{(u_{n},v_{n})\}_{n=1}^{\infty}$
    of pairs of unitaries in $\tilde{A}$ and decreasing sequence of ideals $(I_{n})_{n=1}^{\infty}$ of $A$ as in the proof of \cite[Theorem 3.1]{MR1234394}.
    To get a required irreducible representation,
    it is enough to show that there is a pure state $\omega$ of $A$ and $\delta_{1}>0$ such that
    \begin{equation}\label{ineq:BEK}
    \sup\{\mathrm{Re}\;\omega(u_{n}bv_{n})\mid b\in B, \|b\|=1\}\geq\delta_{1} 
    \end{equation}
    holds for every $n$.
   To get $\omega$,
   we can use a strategy similar to that used in the proof of  \cite[Theorem 3.1]{MR1234394}.
   However, we need to revise the construction of the decreasing sequence $(e_n)_{n}$ of positive elements as follows.
   
    Using $\delta>0$ in the assumption,
    set continuous functions $f\colon[-2,2]\to[0,1]$ and $g\colon[0,1]\to[0,1]$ as follows:
    \begin{align*}
  f(t)&:=
  \begin{cases}
    0 & \text{if $t<\frac{\delta}{3}$,} \\
    \frac{6}{\delta}t-2                 & \text{if $\frac{\delta}{3}\leq t\leq\frac{\delta}{2}$,} \\
   1       & \text{if $\frac{\delta}{2}<t$,}
  \end{cases}\\
  g(t)&:=
  \begin{cases}
      (1-\frac{\delta}{12})^{-1}t&\text{if $t<1-\frac{\delta}{12}$,}\\
      1&\text{if $t\geq1-\frac{\delta}{12}$.}
  \end{cases}
\end{align*}
Let $e_{1}$ and $e_{1}^{\prime}$ be positive norm-one elements of $I_{1}$ such that $e_{1}e_{1}^{\prime}=e_{1}^{\prime}$ holds.
Using our assumptions,
there exists a norm-one element $b_{1}\in B$ such that
\[\sup\mathrm{Spec}(e_{1}^{\prime}(u_{1}b_{1}v_{1}+v_{1}^{*}b_{1}^{*}u_{1}^{*})e_{1}^{\prime})\geq \frac{\delta}{2}.\]
Set $z_{1}:=u_{1}b_{1}v_{1}+v_{1}^{*}b_{1}^{*}u_{1}^{*}$ and $c_{1}:=\|e_{1}^{\prime}f(e_{1}^{\prime}z_{1}e_{1}^{\prime})e_{1}^{\prime}\|^{-1}$.
Let $e_{2}:=g(c_{1}e_{1}^{\prime}f(e_{1}^{\prime}z_{1}e_{1}^{\prime})e_{1}^{\prime})$,
then $e_{2}e_{1}=e_{2}$ holds.
Since $A$ is prime,
the intersection of an ideal $I_{2}$ and a nonzero hereditary subalgebra $\{a\in A\mid ae_{2}=a=e_{2}a\}$ is not equal to zero.
Hence,
we can take a positive norm-one element $e_{2}^{\prime}$ of $I_{2}$ such that $e_{2}^{\prime}e_{2}=e_{2}^{\prime}$holds.
Repeating the same discussion,
we get sequences $(e_{n})_{n}$,
$(e_{n}^{\prime})_{n}$,
$(b_{n})_{n}$ and $(z_{n})_{n}$ such that the following hold for each $n$:
\begin{itemize}
    \item Each $e_{n}$ and $e_{n}^{\prime}$ are positive norm-one elements satisfying $e_{n}\in I_{n-1}$,
    $e_{n}^{\prime}\in I_{n}$, and $e_{n}^{\prime}e_{n}=e_{n}^{\prime}$.
    \item Each $b_{n}$ is a norm-one element of $B$ such that
    \[\sup\mathrm{Spec}(e_{n}^{\prime}(u_{n}b_{n}v_{n}+v_{n}^{*}b_{n}^{*}u_{n}^{*})e_{n}^{\prime})\geq \frac{\delta}{2}.\]
    \item Setting $z_{n}:=u_{n}b_{n}v_{n}+v_{n}^{*}b_{n}^{*}u_{n}^{*}$ and $c_{n}:=\|e_{n}^{\prime}f(e_{n}^{\prime}z_{n}e_{n}^{\prime})e_{n}^{\prime}\|^{-1}$,
    we have
    \[e_{n+1}=g(c_{n}e_{n}^{\prime}f(e_{n}^{\prime}z_{n}e_{n}^{\prime})e_{n}^{\prime}).\]
\end{itemize}

By construction,
we get $e_{n}e_{n+1}=e_{n+1}$.
Since a decreasing sequence \[(K_{n}:=\{\varphi\mid\varphi\;\text{ is a state of $A$ with }\varphi(e_{n})=1\})_{n}\] consists faces in the state space of $A$,
we get a pure state $\omega$ in the set of extremal point of $\cap_{n}K_{n}$.
Using $\omega$,
we get 
\begin{align*}
    \mathrm{Re}\;\omega(u_{n}b_{n}v_{n})&=\omega(z_{n})=\omega(e_{n+1}z_{n}e_{n+1})\\
    &=\omega(g(c_{n}e_{n}^{\prime}f(e_{n}^{\prime}z_{n}e_{n}^{\prime})e_{n}^{\prime})z_{n}g(c_{n}e_{n}^{\prime}f(e_{n}^{\prime}z_{n}e_{n}^{\prime})e_{n}^{\prime}))\\
    &\geq c_{n}^{2}\omega(e_{n}^{\prime}f(e_{n}^{\prime}z_{n}e_{n}^{\prime})e_{n}^{\prime}z_{n}e_{n}^{\prime}f(e_{n}^{\prime}z_{n}e_{n}^{\prime})e_{n}^{\prime})-\frac{\delta}{6}\\
    &\geq\frac{\delta}{3} c_{n}^{2}\omega(e_{n}^{\prime}f(e_{n}^{\prime}z_{n}e_{n}^{\prime})^{2}e_{n}^{\prime})-\frac{\delta}{6}\\
    &\geq \frac{\delta}{3} c_{n}^{2}\omega(e_{n}^{\prime}f(e_{n}^{\prime}z_{n}e_{n}^{\prime})e_{n}^{\prime2}f(e_{n}^{\prime}z_{n}e_{n}^{\prime})e_{n}^{\prime})-\frac{\delta}{6}\\
    &\geq\frac{\delta}{3} \left(\omega(e_{n}^{2})-\frac{\delta}{6}\right)-\frac{\delta}{6}=\frac{\delta}{9}
\end{align*}
for every $n$.
Setting $\delta_{1}:=\frac{\delta}{9}$,
we get the required inequality (\ref{ineq:BEK}).
\end{proof}




 \bibliography{bibliography} 

\end{document}